\documentclass{amsart}
\usepackage{amssymb,graphicx,tikz,pgfplots}
\usetikzlibrary{math} 
\usetikzlibrary{patterns}
\usetikzlibrary{intersections, pgfplots.fillbetween}
\definecolor{imayou}{RGB}{154, 154, 235}
\definecolor{usuai}{RGB}{9, 150, 126}
\definecolor{persred}{RGB}{154,63,63}
\definecolor{sand}{RGB}{201, 177, 60}
\usepackage{xcolor}
\usepackage{bbm,mathrsfs}
\pgfplotsset{compat=1.18}
\usepackage{mathtools}
\mathtoolsset{showonlyrefs=true}
\usepackage{hyperref}

\newcommand{\ep}{\varepsilon}
\renewcommand{\d}{\delta}
\newcommand{\norm}[1]{\left\Vert #1 \right\Vert}

\newcommand{\Fc}{\mathcal{F}}
\newcommand{\arc}{\text{Arc}}
\newcommand{\ccap}{\text{Cap}}
\newcommand{\ltwo}[1]{\norm{#1}_{L^2}}

\newcommand{\lp}[2]{\norm{#1}_{L^{#2}}}

\newcommand{\abs}[1]{\left\vert #1 \right\vert}
\newcommand{\ip}[2]{\left\langle #1,#2 \right\rangle}

\newcommand{\Rb}{\mathbb{R}}
\newcommand{\ra}{\rightarrow}

\newcommand{\<}{\left<}
\renewcommand{\>}{\right>}
\newcommand{\Lc}{\mathcal{L}}

\newcommand{\p}{\partial}

\newcommand{\supp}{\text{supp }}

\newcommand{\Ac}{\mathcal{A}}
\newcommand{\Acr}{\mathcal{A}}
\newcommand{\pAcr}{\mathcal{A}}

    \newcommand{\ti}{\widetilde}

\DeclareMathOperator{\dist}{dist}
\newcounter{thmabc}
\newtheorem{theorem}{Theorem}

\newtheorem{lemma}{Lemma}[section]
\newtheorem{corollary}{Corollary}[theorem]
\newtheorem{proposition}[lemma]{Proposition}
\theoremstyle{definition}
\newtheorem{remark}[lemma]{Remark}
\newtheorem{question}{Question}
\newtheorem{definition}[lemma]{Definition}

\numberwithin{equation}{section}
\title{Observability of Schr\"odinger equations in Euclidean space}
\author{Walton Green}
\author{Perry Kleinhenz}

\begin{document}
\begin{abstract}
    In this paper we introduce a new dynamical condition, the comb geometric control condition, which is sufficient for observability of the Schr\"odinger equation in Euclidean space. We provide examples which show this condition is strictly weaker than the observation set being open and periodic. We also prove for the fractional Schr\"odinger equation that for observation functions which are uniformly continuous, the geometric control condition is equivalent to observability and implies arbitrary time observability. The proofs rely on uncertainty principles for frequency localized functions which are proved using a semiclassical propagation of singularities approach. 
\end{abstract}
\maketitle

\section{Introduction}
Consider the Schr\"odinger equation on $\mathbb R^d$:
	\begin{equation}\label{e:schro} iu_t -\Delta u=0, \quad u(\cdot,0)=u_0.\end{equation}
We are motivated by the question of observability of this equation, namely for which sets $E$ does there exists $C,T>0$ such that the observability inequality
	\begin{equation}\label{e:ob} \int_{\mathbb R^d} |u(x,0)|^2 \, dx \le C \int_0^T \int_E |u(x,t)|^2 \, dx \, dt\end{equation}
	holds for all initial data $u_0$ in $L^2$ and $u$ satisfying \eqref{e:schro}? This problem has been very well-studied in the setting of compact manifolds, most comprehensively for the torus $\mathbb T^d$ \cite{Har89, komornik92, Jak97, Mac10, BZ12, BBZ13, AM14,BZ19, BH2025Rough}. At this point, it is well-known that any open subset of $\mathbb T^d$ is an observation set, and in fact the observation time $T$ can be any (small) positive value \cite{AM14}. Furthermore, in this case many results continue to hold with the presence of a potential \cite{BZ12, BBZ13, AM14, BZ19, BH2025Rough}.
    
    In noncompact settings such as $\mathbb R^d$, the known results are not as strong. However, in the special case of $\mathbb R$, the problem is completely solved \cite{HWW22,SSY25}, and observability is equivalent to $E$ being \textit{relatively dense}, which means 
		\[ \inf_{x \in \mathbb R} |E \cap [x,x+L]| > 0 \]
for some $L>0$. Furthermore, in this case the control time $T$ can be arbitrarily small, and one may add a bounded potential. When $d>1$ however, the results are not as decisive \cite{Wun17, WWZ19, MPS21,Tau23,LBM23, FGW26, LBY26} all of which provide different classes of sets $E$ for which \eqref{e:ob} does hold. See also \cite{HWW22, Pro25} for the case of unbounded potentials. Our goal herein is to provide a systematic study of the observability inequality \eqref{e:ob} for the Schr\"odinger equation on $\mathbb R^d$.
	
Our first stop is to explore the fractional version on \eqref{e:schro}, which, as we will demonstrate, is simpler. Let $\beta$ be a parameter between $0$ and $1$ and we consider the fractional equation
		\[ i u_t = (-\Delta)^{\frac{\beta+1}{2}}u.\] 
We will write $S_\beta(t)$ for the semigroup $\exp(-i (-\Delta)^{\frac{\beta+1}{2}}t)$. In this setup, when $\beta=1$ we recover the Schr\"odinger equation \eqref{e:schro} and when $\beta=0$ we obtain the so-called half wave equation. The half wave equation is well-known \cite{Mil12} to be observable precisely from sets satisfying the Geometric Control Condition (GCC) \cite{BJ16,RT74} as long as the observation set has some minimal regularity, which we will discuss later. To avoid this technicality at this point in the discussion, let us rephrase fractional observability in terms of functions $a$. We say a function $a$ satisfies the GCC if there exists $\overline{C}, L>0$ such that for all line segments $\ell$ of length $|\ell| \geq L$, the line integrals
	\[ \frac{1}{|\ell|} \int_\ell a \, ds \geq \overline{C}. \]

\begin{theorem}\label{thm:gcc}
Let $a:\mathbb R^d \to [0,1]$ be a uniformly continuous function and $\beta \in [0,1)$. The following are equivalent
\begin{itemize}
\item There exist $C,T>0$ such that
	\begin{equation}\label{e:frac-ob}
    \int_{\mathbb R^d} |u_0(x)|^2 \, dx \le C \int_0^T \int_{\mathbb R^d}  a(x) |S_\beta(t)u_0(x)|^2 \, dx.
    \end{equation}
\item $a$ satisfies the GCC.
\end{itemize}
Furthermore, for $\beta \in (0,1]$ if $a$ satisfies the GCC then there exists $C>0$ such that for all $T>0$
\begin{equation}\label{e:frac-ob-arb}
    \int_{\mathbb R^d} |u_0(x)|^2 \, dx \le C \exp\left(C {T^{-\frac{2}{\beta}}} \right)\int_0^T \int_{\mathbb R^d} a(x) |S_\beta(t)u_0(x)|^2  \, dx,
\end{equation}
that is the fractional Schr\"odinger equation is observable from $a$ in arbitrary time. 
\end{theorem} 
As pointed out in \cite[Proposition 2.11]{MPS21}, the results \cite[Theorem 1]{GJM22} and \cite[Corollary 2.17]{Mil12} together show that for uniformly continuous $a$ the geometric control condition implies observability for the fractional Schr\"odinger equation on $\Rb^d$. We improve on this by showing the GCC is necessary when $a$ is uniformly continuous and by providing an explicit formula for the observability cost. On compact manifolds, if an open set satisfies the GCC then it observes the fractional Schr\"odinger equation and this is still true with an $L^{\infty}$ potential \cite[Theorem 1]{Mac21}. In that same paper it was shown that on $\mathbb{S}^d$, the GCC is necessary for observability of the fractional Schr\"odinger equation. However to our knowledge, when $\beta>0$ a result providing a non-trivial necessary condition on general manifolds is not known.

Our proof of this theorem goes through an \textit{a priori} weaker condition than GCC, which is a type of density defined over long, thin rectangles of dimensions $\sim \lambda^\beta \times \lambda^{\frac{\beta-1}{2}}$, with $\lambda$ large. See the start of Section \ref{s:fractional} for the precise statement. Without the condition that $a$ be uniformly continuous, this condition is necessary for \eqref{e:frac-ob} and strictly weaker than GCC. But as it stands, very few observability results for wave and Schr\"odinger equations hold without regularity constraints on the observation sets/functions; see notable exceptions \cite{Rou24,BG20,LBY26,LBM23,BZ19,JS25}. We hope our theorem makes it clear that the imposition of uniform continuity blurs the geometric differences between the different values of $\beta \in [0,1)$.

The open and shut nature of the case $\beta \in [0,1)$ focuses our attention on the case $\beta=1$. Developing the propagation of singularities technique from compact manifolds in the unbounded Euclidean setting we obtain a new geometric condition, which we call comb GCC. This condition is succinctly explained by taking an observation function $a:\mathbb R^2 \to [0,1]$ which is periodic. Then in general the GCC fails along the horizontal and vertical directions. However, for a relatively dense subset of horizontal or vertical lines, the GCC is satisfied, see Figure \ref{fig:combGCCintro}. In this case we would say that $a$ satisfies the comb GCC in the vertical and horizontal directions.

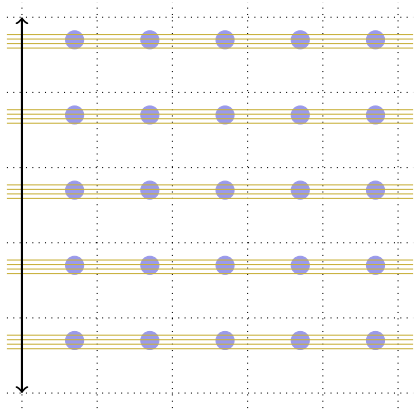
\begin{figure} \label{fig:combGCCintro}
\begin{tikzpicture}
	\draw[dotted] (-.2, -.2) grid (5.2,5.2);
	\foreach \x in {0.5,1.5,...,4.5}{
  	\foreach \y in {0.5,1.5,...,4.5}{
    \filldraw[imayou] (\x+0.2,\y+0.2) circle (0.12);
  }   
}
	\foreach \x in {1,3,5,7}{
  	\foreach \y in {0.5,1.5,...,4.5}{
    \draw[sand] (-.2,-0.03*\x+\y+0.3) -- (5.2,-0.03*\x+\y+0.3);
  }   
}
\draw[thick,<->] (0,0) -- (0,5);
\end{tikzpicture}
\caption{A periodic set which satisfies the GCC along the yellow horizontal lines. In the vertical direction, the yellow lines satisfy the GCC.}
\end{figure}

Another way to understand the comb GCC in a particular direction, is that the function may be ``smeared'' by a finite amount in said direction to form a function that is invariant in that direction, and satisfies the GCC in the perpendicular direction; see the non-example in Figure \ref{fig:introNon} below. Our two main families of examples: periodic sets and  certain product-type sets satisfy this smearing condition along a finite collection of directions, and in fact in all other dirctions satisfy the GCC for large $L$ which may depend on the particular direction. In the compact setting, one could use in the infinite speed of propagation principle of defect measures to effectively allow $L=\infty$. In our non-compact setting, $L$ may still increase without bound over various directions. However, we are only able to obtain observability resolvent estimates for wavepackets frequency localized to a sector of width $o(L^{-1})$; see \eqref{e:comb-eff} for the precise limitation which this imposes.

In summary, for now, the comb GCC captures, quantitatively, the failure of the GCC in two admissible ways:
\begin{itemize}
    \item[(A1)] The GCC may fail in certain directions as long as the set can be smeared by some finite amount to generate an invariant set satisfying GCC in the perpendicular direction. 
    \item[(A2)] The GCC length $L$ may blow-up (in a controlled way) as one approaches particular ``bad'' directions. 
\end{itemize}
Our second main result is this new sufficient condition for observability in the case $\beta=1$.

\begin{theorem}\label{thm:effect}
If $a: \Rb^d \ra [0,1]$ is uniformly continuous and satisfies the comb GCC, then there exist $C,T>0$ such that
	\[ \int_{\mathbb R^2} |u(x,0)|^2 \, dx \le C \int_0^T \int_{\mathbb R^2} |a(x) u(x,t)|^2 \, dx \, dt \]
holds for all $u$ satisfying \eqref{e:schro}.
\end{theorem}

This condition, while we do not believe it to be necessary on $\mathbb R^d$, is in our opinion the analogue of the ``any open set'' condition from $\mathbb T^d$. We support this by showing that any non-trivial periodic function $a$ on $\Rb^2$ indeed satisfies the comb GCC. That Schr\"odinger observability holds for a periodic function is already contained implicitly in \cite{Wun17} and explicitly in \cite{Tau23,LBM23}. In those approaches, one periodizes the problem on $\mathbb R^d$ and directly applies the $\mathbb T^d$ result \cite{jaffard90,komornik92}.

Moreover, to demonstrate the power of the comb GCC, we also provide a class of sets satisfying the comb GCC which were until now unknown to be observation sets.
\begin{proposition}\label{p:product-intro}
If $E,F \subset \mathbb R$ satisfy
	\[ \inf_{x \in \mathbb R} |E \cap [x,x+1]| + \inf_{y \in \mathbb R} |F \cap [y,y+1]| > 1.\]
Then $\mathbbm 1_{E \times F}$ satisfies the effective comb GCC. In particular, if $a \geq \mathbbm 1_{E \times F} $ is uniformly continuous, then the Schr\"odinger equation is observable from $a$.
\end{proposition}

As mentioned above, we do not know that comb GCC is necessary for Schr\"odinger observability. At this time,  the only known necessary geometric condition on the observation function $a :\mathbb R^d \to [0,1]$ that is necessary for Schr\"odinger observability to hold is that $a$ be relatively dense which means there exists $L,\eta>0$ such that
    \[ \frac{1}{L^d} \int_Q a(z) \, dz \ge \eta \]
for all cubes $Q \subset \mathbb R^d$ of side length $L$, \cite[Theorem 2.6]{MPS21} and \cite[Remark 2.2]{HWW22}. By rescaling $a$ and $u$, we can work with $L=1$ when it is more convenient (as in Proposition \ref{p:product-intro} above). Relative density is a low-frequency phenomenon, which is in fact necessary for nearly any observation inequality which possesses translation invariance. While relative density is a weak condition, it has been shown to be necessary and sufficient for both the observability of the heat equation in $\mathbb R^d$ \cite{EV18,WWZZ19} as well as for logarithmic energy decay of the damped wave equation on unbounded manifolds with damping $a$ \cite{BJ16,IS23}.

\begin{question}\label{q:1}
    Is the Schr\"odinger equation observable by any bounded, uniformly continuous $a$ which is relatively dense?
\end{question}

Experts will realize that an affirmative answer to Question \ref{q:1} would also affirmatively answer the following question.

\begin{question}\label{q:2}
Does the energy of damped wave equation on $\mathbb R^d$ with relatively dense damping decay polynomially?
\end{question}
Let us conclude our introduction by giving an example of a relatively dense set that does not satisfy comb GCC. Consider the sets $E_{\text{L}}$ and $E_{\text{R}}$ which are the left (resp. right) half of all vertical integer strips. Then, create $E$ to be $E_{\text{L}}$ on the upper half plane and $E_{\text{R}}$ on the lower half plane; see Figure \ref{fig:introNon}. Then, as Figure \ref{fig:introNon} demonstrates, principle (A1) is not satisfied--not only does the GCC fail in the vertical direction, but the comb GCC fails. As we will see below, we could salvage the vertical direction $\theta_0=(0,1)$ if it was $o(L^{-1})$ distance from another direction $\theta_1$ which satisfied the GCC in time $L$. But notice that for directions $\ep$ away from vertical the GCC is satisfied in precisely time $\ep^{-1} \ne o(\ep^{-1})$. This in the precise sense in which (A2) is also not satisfied. Therefore, this example cannot satisfy our comb GCC, but it is clearly relatively dense.

\begin{figure}

\begin{tikzpicture}[scale=0.6]
\begin{scope}
	\foreach \x in {0.5,1.5,...,7.5}{
    \filldraw[imayou] (\x,4) rectangle (\x-0.5,8.2);
    \filldraw[imayou] (\x,4) rectangle (\x+0.5,-0.2);
  }   
	\draw[dotted] (-.2, -.2) grid (8.2,8.2);
\end{scope}
\begin{scope}[xshift=10cm]
	\foreach \x in {0.5,1.5,...,7.5}{
    \filldraw[imayou] (\x,4) rectangle (\x-0.5,8.2);
    \filldraw[imayou] (\x,7) rectangle (\x+0.5,-0.2);
  }  
	\draw[dotted] (-.2, -.2) grid (8.2,8.2);
\end{scope}
\end{tikzpicture}
\caption{A relatively dense set $E$ which cannot be smeared a finite amount in the vertical direction to create an invariant set.} \label{fig:introNon}
\end{figure}
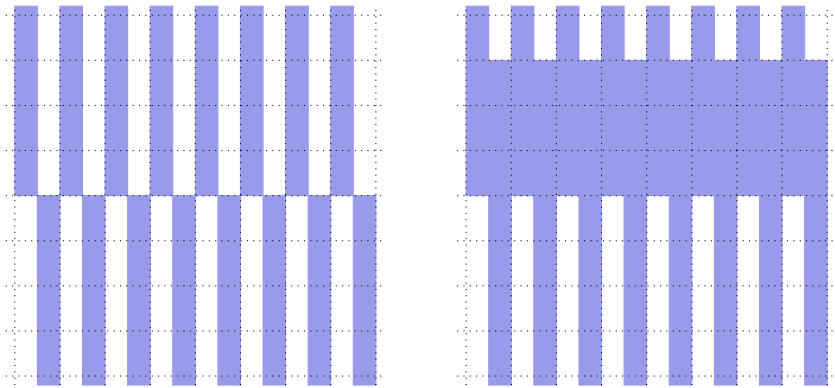

\subsection{Literature Review}
On $\Rb^d$ for uniformly continuous $a$, a  sufficient condition for Schr\"odinger observability for any (small) $T>0$, is that $a$ satisfies the geometric control condition \cite[Theorem 1.2]{BJ16} and \cite[Theorem 9]{Tau23}. We note that every $a$ which satisfies the geometric control condition also satisfies our comb GCC, as we show in Lemma \ref{l:rd-comb}.
    
    \begin{remark} We point out that these results can straightforwardly be extended to add a bounded potential $V$ and compute an explicit observability cost using our Theorem \ref{thm:resolventArbObs}. 
    \end{remark}

A necessary condition for any $a$ for Schr\"odinger observability is that $a$ be relatively dense \cite[Theorem 2.6]{MPS21} and \cite[Remark 2.2]{HWW22}.

When $d=1$ relative density is necessary and sufficient for Schr\"odinger observability \cite[Theorem 2.7]{MPS21} and \cite[Theorem 1.1]{HWW22}. And in fact, relative density is sufficient for arbitrary time observability \cite[Theorem 1.2]{SSY25}. 
It is worth emphasizing that in $d=1$ relative density and the geometric control condition coincide.

There are conditions weaker than the geometric control condition which guarantee Schrodinger observability. On $\Rb^d$ if $a$ is periodic and its support contains an open set, then finite time Schr\"odinger observability holds \cite{Wun17}, \cite[Theorem 2]{Tau23}. We show that every such $a$ satisfies our comb GCC in Proposition \ref{p:period}.

On $\Rb^2$ the assumption on $a$ can be relaxed to merely be periodic and positive on a measurable set, and a bounded potential can be added \cite[Theorem 1.2]{LBM23}. We point out that the analogous results hold for tori based on their dimension. That is measurable sets can be used to observe the Schr\"odinger equation on $\mathbb{T}^2$ \cite{BZ19}, but on $\mathbb{T}^d$ it is only known that open sets observe the Schr\"odinger equation \cite{AM14}. 

For non-periodic, non GCC results: \cite[Theorem 1.1, Remark a5]{WWZ19} shows that if $\{a=0\}$ is a finite measure set, then the Schr\"odinger equation is observable from $a$. Recently \cite{LBY26} have provided a sufficient condition which is a strengthening  of relative density that requires the set to be dense at finer scales towards infinity. 

\subsection{Paper Outline}
In Section \ref{s:fractional} we prove that the geometric control condition is equivalent to observability for the fractional Schr\"odinger equation, Theorem \ref{thm:gcc}. In Section \ref{s:fracPLSProof} we prove an uncertainty principle for functions supported on an annulus, Theorem \ref{thm:frac} which is a key ingredient in the proof of Theorem \ref{thm:gcc}.

In Section \ref{s:combGCC2} we introduce the comb GCC in two dimensions and use it to prove Theorem \ref{thm:comb} in two dimensions. In Section \ref{s:product} we prove that products of relatively dense sets with sufficient density satisfy the comb GCC, Proposition \ref{p:product-intro}. In Section \ref{s:periodic} we prove that in $\Rb^2$ nontrivial periodic sets satisfy the comb GCC. In Section \ref{s:higher} we inductively define the comb GCC in higher dimensions. 

In Section \ref{s:comb-proof} we inductively prove an observability resolvent estimate, Proposition \ref{p:main}, via a similar estimate for frequency localized functions, Proposition \ref{p:main-microlocal}. It is then straightforward to conclude Theorem \ref{thm:comb}.

In Section \ref{s:resolventConnection} we state two results connecting observability estimates to observability resolvent estimates. We prove one of these results, Theorem \ref{thm:resolventArbObs}, which states that a decaying in $\lambda$ observability resolvent estimate is sufficient for arbitrary time observability.

\textbf{Acknowledgments} The authors would like to thank Nicolas Burq, Patrick G\'erard, Jean Lagac\'e, Matthieu L\'eautaud, Chenmin Sun, and Jared Wunsch for helpful conversations. 

\textbf{Funding:} This work was partially supported by the National Science Foundation [DMS-2530465 to P.K.] and Illinois State University [New Faculty Initiative Grant to P.K.].

\section{Fractional case}\label{s:fractional}
Given $s,t>0$, we denote the anisotropic dilation $D_{s,t}$ by
        \begin{equation} \label{eq:dilationdef}
             D_{s,t}(z) = (sz_1,sz_2,\ldots, sz_{d-1}, tz_d) = (sz', tz_d).
        \end{equation}
Note that $D_{s,t}$ applied to a cube is a rectangle, and this is always what we will  mean by rectangles. Let $\mathsf R_{\beta}^L$ be the collection of all rectangles with one side of length $L\lambda^\beta$ and the rest of length $L\lambda^{\frac{\beta-1}{2}}$ for $\lambda \ge 1$. To state things precisely, we introduce the shorthand for cubes
    \begin{equation}\label{e:cubeShorthand} [z,z+L]^d = [z_1,z_1+L] \times \cdots \times [z_d,z_d+L]. \end{equation}
Furthermore, given $\theta \in \mathbb S^{d-1}$ we use $\theta$ to denote the rotation action on $\mathbb R^d$ which sends $ (0,\ldots,0,1) \mapsto \theta $. Then,
    \[ \mathsf{R}_\beta^L = \left\{ \theta \left(D_{s,t} [z,z+L]^d\right) : z \in \mathbb R^d, \ \theta \in \mathbb S^{d-1}, \ s=L\lambda^{\frac{\beta-1}{2}}, \ t = L\lambda^\beta\right\}.\]
Given a function $a:\mathbb R^d \to [0,1]$, we define the lower $(\beta,L)$ density
        \[ a_{\beta,L} = \inf_{R \in \mathsf R_\beta^L} \frac{1}{|R|} \int_R a(x) \, dx.\]
    Notice that $a_{0,L} >0$, is equivalent to $a$ satisfying the GCC and $a_{1,L} >0$, is equivalent to $a$ being relatively dense. Furthermore, a covering argument shows that we could equivalently consider the larger class of rectangles with side lengths \textit{at least} $L\lambda^\beta$ and $L\lambda^{\frac{\beta-1}{2}}$.

    We will connect $a_{\beta,L}$ to uncertainty principles for functions supported in the annulus
        \[ A_{\beta,\lambda} = \{ \xi \in \mathbb R^d : \lambda-\d \lambda^{-\beta} \le |\xi| \le \lambda + \delta\lambda^{-\beta} \}.\]
Throughout, $\delta$ is a fixed positive number. 
    \begin{theorem}\label{thm:frac}
         Let $a: \Rb^d \ra [0,1]$ be uniformly continuous and $\beta \in [0,1)$. $a_{\beta,L} >0$ if and only if there exists $C>0$ such that
            \begin{equation} \norm{u}_{L^2} \le C \norm{au}_{L^2}\end{equation}
        for all $\lambda \ge 1$ and all $u$ with $\supp \hat u \subset A_{\beta,\lambda}$.
    \end{theorem}

    Our proof strategy is similar to \cite{GJM22} which covers the case $\beta=0$. We prove Theorem \ref{thm:frac} in Section \ref{s:fracPLSProof}.

It is well-known by now that such uncertainty principle inequalities (elsewhere sometimes called wavepacket estimates \cite{Mil12}) are equivalent to resolvent estimates \cite{BZ04,Gre20}. The following lemma states a general version of this principle for applications of statements like Theorem \ref{thm:frac} to Schr\"odinger observability inequalities.

In particular, taking $\gamma=\beta+1$ in Lemma \ref{l:equiv}, and using the well-known equivalence between long-time observability and resolvent estimates (see Appendix), we obtain that Theorem \ref{thm:frac} states that $a_{\beta,L}>0$ is equivalent to large-time fractional Schr\"odinger observability \eqref{e:frac-ob}. Subsequently, Theorem \ref{thm:gcc} will follow by relating $a_{\beta,L}$ to the GCC in section \ref{s:gcc}. 
\begin{lemma}\label{l:equiv}
Let $a: \Rb^d \ra [0,1]$. 
If there exists $\d, \beta, C>0, \lambda_0 \geq 1$ such that for all $\lambda \geq \lambda_0$ and $u$ with $\supp \hat{u} \subset A_{\beta,\lambda}$ we have 
\begin{equation} \label{e:uncp}
    \|u\|_{L^2} \leq C \|a^{1/2} u\|_{L^2},
\end{equation}
then for all $\gamma>0$ there exists $B>0$ such that for all $u \in H^{\gamma}$, $\lambda \geq \lambda_0$
we have
	\begin{equation}\label{e:resolve}
            \ltwo{u} \leq C \ltwo{a^{1/2} u} + CB \lambda^{1+\beta-\gamma} \ltwo{\left( (-\Delta)^{\frac \gamma 2}-\lambda^\gamma \right)u}.
        \end{equation}
On the other hand, if for some $C,B, \beta, \gamma>0$, $\lambda_0 \geq 1$ \eqref{e:resolve} holds for $\lambda \geq \lambda_0$, then there exists $\delta>0$ such that \eqref{e:uncp} holds for $\lambda \geq \lambda_0$.
\end{lemma}

\begin{proof}
Assume \eqref{e:uncp} holds. Fix $\lambda,\gamma$ and let 
\begin{equation}
    A^*_\ep = \{ \lambda^\gamma-\ep \lambda^{\gamma-\beta-1} \le |\xi|^{\gamma} \le \lambda^{\gamma} + \ep\lambda^{\gamma-\beta-1}\}.
\end{equation}
We will show $A^*_{\ep} \subset A_{\beta,\lambda}$ for suitable $\ep>0$. Let us first remind the reader of the standard derivation of the forward direction of the Lemma from that containment. 
Suppose $u \in H^{\gamma}$, then by the containment $A^*_{\ep} \subset A_{\beta,\lambda}$, $u=g+f$ where $\hat g$ is supported in $A_{\beta,\lambda}$ and $\hat f = \mathbbm{1}_{(A^*_{\ep})^c} \hat u$. Then, by two applications of the triangle inequality, and the assumed uncertainty principle \eqref{e:uncp} for $g$,
	\begin{equation}\label{e:triangle} \ltwo{u} \le \ltwo{f} + \ltwo{g} \le \ltwo{f} + C\ltwo{ag} \le \left(C +1\right) \ltwo{f} + C \ltwo{au}.\end{equation}
Now since $\hat f = \mathbbm{1}_{(A^*_{\ep})^c} \hat u$, applying Plancherel's theorem twice 
	\begin{equation}\label{e:Delta-gamma}
    \ep \lambda^{\gamma-\beta-1} \norm{f}_{L^2} \leq \norm{ \left((-\Delta)^{\frac{\gamma}{2}}-\lambda^\gamma\right) f }_{L^2} \le \norm{ \left((-\Delta)^{\frac{\gamma}{2}}-\lambda^\gamma\right) u }_{L^2} .\end{equation}
Combining this with the previous inequality gives  
\begin{equation}
    \ltwo{u} \leq \frac{C+1}{\ep} \lambda^{1+\beta-\gamma} \ltwo{\left( (-\Delta)^{\frac{\gamma}{2}} - \lambda^{\gamma}\right) u} + C \ltwo{au},
\end{equation}
as desired. 

Now, to show $A^*_{\ep} \subset A_{\beta,\lambda}$, notice that for $\xi \in A^*_{\ep}$ clearly $(\lambda^{\gamma}-\ep \lambda^{\gamma-\beta-1})^{\frac{1}{\gamma}} \le |\xi| \le (\lambda^{\gamma}+\ep\lambda^{\gamma-\beta-1})^{\frac{1}{\gamma}}$. Then, by Taylor expansion
	\[ (\lambda^{\gamma}\pm \ep\lambda^{\gamma-\beta-1})^{\frac{1}{\gamma}} = \lambda \pm \frac{\ep \lambda^{\gamma-\beta-1}}{\gamma}t^{\frac{1-\gamma}{\gamma}}, \quad \lambda^{\gamma} - \ep \lambda^{\gamma-\beta-1} \le t \le \lambda + \ep\lambda^{\gamma-\beta-1}.\]
Since $\lambda \ge 1$ and $\frac{\gamma-\beta-1}{\gamma}\le 1$, as long as $\ep<\frac{1}{4}$ we have $\frac 12 \lambda^{\gamma} \le t \le 2\lambda^{\gamma}$ and hence
	\[ t^{\frac{1-\gamma}{\gamma}} \le 2^{\abs{\frac{1-\gamma}{\gamma}}} \lambda^{1-\gamma} .\] 
So now choosing $0<\ep< \frac 14$ such that
	\[ \frac{\ep 2^{\abs{\frac{1-\gamma}{\gamma}}}}{\gamma} \le \delta, \]
we obtain
	\[ \Bigl| |\xi|-\lambda \Bigr| \leq \d \lambda^{-\beta}.\]
Thus $A^*_{\ep} \subset A_{\beta,\lambda}$.

To show the other direction, we choose $\ep_0< \frac{1}{2CB}$ so that if $\hat u$ is supported in 
	\[ A_* = \{ \lambda^{\gamma}- \ep_0\lambda^{\gamma-\beta-1} \le |\xi|^\gamma \le \lambda^{\gamma} + \ep_0 \lambda^{\gamma-\beta-1}\} \]
and \eqref{e:resolve} holds, then
	\[ \ltwo{u} \le C\ltwo{au} + \frac{1}{2} \ltwo{u}.\]
Thus we can absorb the second term on the right back into the left hand side and obtain \eqref{e:uncp}.
Now, by the same steps that were used to show the containment $A^*_{\ep} \subset A_{\beta,\lambda}$, we may choose $\delta$ small enough that $A_{\beta,\lambda} \subset A_*$.
\end{proof}

 \subsection{Connection to Geometric Control Condition}\label{s:gcc}
  For general $a$ and $\beta \ge 0$, the rectangle condition $a_{\beta,L} \ge \eta$ is weaker than the geometric control condition. 
\begin{lemma}\label{l:gcc-to-rect}
    Suppose $a$ satisfies the GCC with constants $L$ and $\overline{C}$, then for each $\beta \ge 0$, $a_{\beta,L} > \overline{C}$.
\end{lemma}
\begin{proof}
    By rotation and translation invariance of our conditions, to bound $a_{\beta,L}$ from below we may consider only $R=D_{\lambda^{\frac{\beta-1}{2}},\lambda^\beta}[0,L]^d$. 
Then since $a$ satisfies the geometric control condition
    \begin{align}
        \frac{1}{|R|} \int_R a(z) dz &= \frac{1}{|R|} \int_{[0,L\lambda^{\frac{\beta-1}{2}}]^{d-1}} \int_0^{L\lambda^\beta} a(z',z_d) dz_d dz' \\
        & \geq \frac{1}{|R|} \int_{[0,L\lambda^{\frac{\beta-1}{2}}]^{d-1}}  \overline{C}L\lambda^{\beta} dz' \\
        &= 
        \overline{C}.
    \end{align}
\end{proof}
In fact, $a_{\beta,L} >0$ is strictly weaker than the GCC as the following example shows. Let $a(x,y)$ be the indicator function of
    \begin{align}
        E_\beta = \{(x,y): |x|>1, |y| > |x|^{\frac{\beta-1}{2\beta}}\} \cup \{(x,y): |x|<1, |y|^{\frac{\beta-1}{2\beta}}<|x|\} \subset \mathbb R^2.
    \end{align}
    \begin{figure}
        \centering
        \begin{tikzpicture}
        
            \filldraw[imayou,domain=1:12,samples=50] plot (\x,{\x^(-0.5)}) |- (1,1.5);
            \filldraw[imayou,domain=1:12,samples=50] plot (\x,{-1*\x^(-0.5)}) |- (1,-1.5);
            \draw[dotted] (1,-1.5) grid (12,1.5);
            \draw[ultra thick,sand] (1,-0.5) rectangle (10,0.5);
            \draw[ultra thick,persred](1,-1) rectangle (5.5,1);
            \draw (0.5,0) node {$L\lambda^\beta$};
            \draw (6.5,0.8) node {$L\lambda^{\frac{\beta-1}{2}}$};
            \draw (8,-1) node {$|y|^\beta >|x|^{\frac{\beta-1}{2}}$};
        \end{tikzpicture}
        \caption{$E_\beta$ for $\beta=\frac 12$ and $1 < x < 12$, $-1<y<1$. Also pictured are two rectangles from $\mathsf R_{\beta,3}$ $\lambda=\sqrt{3}$ and $\lambda = \sqrt{ \frac 32}$.}
        \label{fig:rectangles}
    \end{figure}
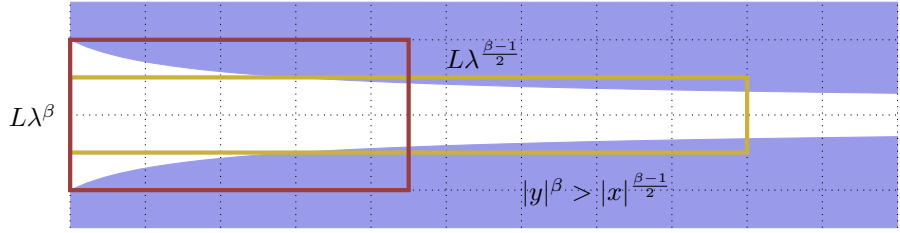
    
    This satisfies $a_{\beta,3}>0$, see Figure \ref{fig:rectangles}, but the $x$ and $y$ axes do not intersect this set, so the GCC does not hold. 
However, our result Theorem \ref{thm:frac} only applies to $a$ which are uniformly continuous. And under this regularity assumption, actually $a_{\beta,L} \ge \eta$ for any $\beta \in [0,1)$ is equivalent to GCC.
  \begin{lemma}\label{l:rect-to-gcc}
  Suppose $a: \Rb^d \ra [0,1]$ is uniformly continuous and satisfies $a_{\beta,L} \ge \eta$ for some $\beta \in [0,1)$ and $L,\eta>0$. Then $a$ satisfies the GCC.
  \end{lemma}
  \begin{proof}
Suppose towards a contradiction that $a_{\beta,L} \ge \eta$ but $a$ fails the GCC. 
By uniform continuity, we can find $\ep \in (0,1)$ such that $|a(x)-a(y)| \le \frac{\eta}{4}$ for $|x-y| \le \ep \sqrt{d-1}$. Since $\beta<1$, we can find $\lambda$ large enough that $L\lambda^{\frac{\beta-1}{2}} < \ep$. 

On the other hand, since $a$ fails the GCC, we can find a line segment $\ell_0$ such that
  	\[ \frac{1}{|\ell_0|} \int_{\ell_0} a \le \frac{\eta}{16}\]
and $|\ell_0| \ge L\lambda^{\beta}$. We will now find an $\ell \subset \ell_0$ with $L\lambda^\beta = |\ell|$ such that the above display almost holds replacing $\ell_0$ by $\ell$. If $\ell_0$ already satisfies this condition we take $\ell=\ell_0$. Otherwise $\ell_0$ contains $J=\lfloor \frac{|\ell_0|}{L\lambda^\beta} \rfloor \ge 1$ disjoint line segments $\ell_1,\dots,\ell_J$ of length $L\lambda^\beta$. One such segment must satisfy
	\begin{equation}\label{e:ellj}\frac{1}{|\ell_j|} \int_{\ell_j} a \le \frac{\eta}{4},\end{equation}
since if all failed the above condition, then
	\[ \int_{\ell_0} a \ge \sum_{j=1}^J \int_{\ell_j} a \ge \frac{\eta}{4} J L \lambda^\beta \ge \frac{\eta}{4} \frac{J+1}{2} L \lambda^\beta \ge \frac{|\ell_0| \eta}{8},\]
	which contradicts the choice of $\ell_0$. Take $\ell$ to be one such $\ell_j$ satisfying \eqref{e:ellj}. 
	Now, letting $R$ be the square neighborhood of $\ell$ of width $L\lambda^{\frac{\beta-1}{2}}$, $R$ indeed belongs to $\mathsf R_{\beta,L}$. 
And, for $x \in R$, let $\pi(x)$ denote the closest point to $x$ in $\ell$. Since $|x-\pi(x)| \leq L \lambda^{\frac{\beta-1}{2}} \sqrt{d-1} < \ep \sqrt{d-1}$, we have
	\[ a(x) \le a(\pi(x)) + \frac{\eta}{4}.\]
Then using \eqref{e:ellj} we have
	\[ \int_{R} a \le \int_{t \in \ell} \int_{x \in \pi^{-1}(t)} a(x) \, dx \, dt\le  (L \lambda^{\frac{\beta-1}{2}})^{d-1} \int_{t \in \ell} \left[ a(t) + \frac{\eta}{4}\right] \, dt \le |R| \frac{\eta}{2}.\]
But this contradicts the fact that $\frac{1}{|R|} \int_R a \ge \eta$.
  \end{proof}
Now we can assemble Theorem \ref{thm:frac} and Lemmas \ref{l:equiv} and \ref{l:rect-to-gcc} to prove Theorem \ref{thm:gcc} on the equivalence between observability of the fractional Schr\"odinger equation and the GCC. We rely heavily on the well-known equivalence between the resolvent estimates and observability; which we recall for the reader's convenience in the Appendix as Theorem \ref{thm:Miller}.
\begin{proof}[Proof of Theorem \ref{thm:gcc}]
For the sake of clarity (at the expense of notation), consider the fractional Schr\"odinger equation with exponent $\frac{\beta_*+1}{2}$ for $\beta_* \in [0,1)$. First, if $0 \le a(x) \le 1$ is uniformly continuous and satisfies the GCC, then by Lemma \ref{l:gcc-to-rect}, $a_{0,L} >0$ for some $L>0$. Then by Theorem \ref{thm:frac} and Lemma \ref{l:equiv} with $\beta=0$ and $\gamma=\beta_*+1$, we obtain the resolvent estimate
	\begin{align}
    \label{e:resolvent-beta-star} 
    &\ltwo{u} \le C \left( \ltwo{au} + \lambda^{-\beta_*} \ltwo{ \left[ (-\Delta)^{\frac{\beta_*+1}{2}} - \lambda^{\beta_*+1} \right] u} \right), \quad \lambda \ge 1 \text{ or }\\
    &\ltwo{u} \leq C \left(\ltwo{au} + \lambda^{-\frac{\beta_*}{\beta_*+1}} \ltwo{((-\Delta)^{\frac{\beta_*+1}{2}}-\lambda)u} \right), \qquad \lambda \geq 1.
    \end{align}
Since $\beta_* \ge 0$, $C \lambda^{-\frac{\beta_*}{\beta_*+1}}$ is uniformly bounded so by Theorem \ref{thm:Miller} and Lemma \ref{l:lowFreqResolve}, we obtain the observability inequality \eqref{e:ob}. To conclude arbitrary time observability we apply Theorem \ref{thm:resolventArbObs} with $\ep = \frac{2\beta_*}{\beta_*+1}$ which lies in $(0,1)$  

as long as $\beta_* \in(0,1)$. 

On the other hand, if we assume observability, then by Theorem \ref{thm:Miller} we have the resolvent estimate 
\begin{equation}
    \ltwo{u} \leq C \left( \ltwo{au} + \ltwo{((-\Delta)^{\frac{\beta_*+1}{2}}- \lambda) u} \right) \quad \text{ for all } u \in H^2, \lambda \in \Rb.
\end{equation}
Applying Lemma \ref{l:equiv} with $\gamma=\beta_{*}+1$ and $\beta= \beta_*$ we obtain the uncertainty principle \eqref{e:uncp}. Then by Theorem \ref{thm:frac} we have $a_{\beta_*,L} >0$ for some $L$. This implies GCC by Lemma \ref{l:rect-to-gcc} since $a$ is uniformly continuous.
\end{proof}

    \begin{remark}\label{r:potential}
    By the triangle inequality, decaying resolvent estimates such as \eqref{e:resolvent-beta-star} for $\beta_*>0$ afford the addition of a bounded potential $V$ for large $\lambda$. Thus, if $a$ is uniformly continuous and satisfies the GCC, then we obtain the high-frequency resolvent estimate
        \begin{equation}\label{e:high-potential} \ltwo{u} \le C \left( \ltwo{au} + \lambda^{-\frac{\beta_*}{\beta_*+1}} \ltwo{ \left[ (-\Delta)^{\frac{\beta_*+1}{2}} + V - \lambda \right] u} \right), \quad \lambda \ge \lambda_0,
        \end{equation}
    for some large $\lambda_0>0$. 
    If we also had the low frequency counterpart to \eqref{e:high-potential}, namely that for each $\lambda_0>0$ there exists $C>0$ such that
        \begin{equation}\label{e:low-potential} \ltwo{u} \le C \left( \ltwo{au} + \ltwo{ \left[ (-\Delta)^{\frac{\beta_*+1}{2}} + V - \lambda \right] u} \right), \quad |\lambda| \le \lambda_0.
        \end{equation}
        then, by Theorem \ref{thm:resolventArbObs} below, the second statement in Theorem \ref{thm:gcc} would hold for fractional Schr\"odinger equations with bounded potentials as well. Note that without a potential, this is an immediate consequence of the PLS theorem (see Lemma \ref{l:lowFreqResolve} below), while if there is a potential and $\beta^*=1$, then the associated heat equation is observable from relatively dense sets \cite{DWZ20,LBM25}, and hence the low-frequency resolvent estimate \eqref{e:low-potential} indeed holds. But for general $\beta^* \in (0,1)$ we believe this to be an open problem.
    \end{remark}
    
    \section{Proof of Theorem \ref{thm:frac}}\label{s:fracPLSProof}
        It will be helpful to frequently switch between $a$ and certain indicator functions. And this is nearly always possible in the following sense. For any $0 \le a(x) \le 1$, $\eta,\ep>0$ and $F \subset \mathbb R^d$:
        	\begin{equation}\label{e:switch} \frac{1}{|F|} \int_F a \ge \eta \implies \frac{1}{|F|} \int_F \mathbbm{1}_{E} \ge \eta-\ep, \quad E = \left\{ x \in \mathbb R^d : a(x) \ge \ep \right\}.\end{equation}
    \begin{proof}[Proof of \eqref{e:switch}]
    	\[ \int_F a = \int_{F \cap E}a + \int_{F\backslash E} a \le \int_F \mathbbm{1}_E + \ep|F|.\]
    \end{proof}

    \subsection{Rescaling PLS Theorem}
    The first step is a simple rescaling of the Paneah-Logvinenko-Sereda Theorem \cite{Pan61,LS74,Kov01,Hav12} which relates the geometric property of relative density to a statement about the Fourier support of a function. 
    \begin{theorem}[PLS Theorem]\label{thm:pls}
        For each $\eta,L,\sigma>0$, there exists $C>0$ such that
            \begin{equation}\label{e:PLS} \norm{u}_{L^2} \le C \norm{au}_{L^2}\end{equation}
        holds for all $u$ satisfying $\supp \hat u \subset [\zeta,\zeta+\sigma]^d$ for any $\zeta$ and all $a$ which are $(L,\eta)$ relatively dense.
    \end{theorem}
    
    Recall anisotropic dilation $D_{s,t}$ as defined by \eqref{eq:dilationdef}. For a function $f$, we define 
        \begin{equation}
            f_{s,t}(z) = \left( s^{d-1}t\right)^{\frac 12} f(D_{s,t} z) \label{eq:dilation}
        \end{equation}
    and note it satisfies 
        \[ \norm{f_{s,t}}_{L^2} = \norm{f}_{L^2}, \quad \Fc( f_{s,t}) = (\hat f)_{\frac 1s,\frac 1t}.\]
    Let $\mathsf R^{L,\bot}_{s,t} = \{ D_{s,t}[z,z+L]^d : z \in \mathbb R^d\}$ denote all rectangles of the form
        \[ [\zeta_1,\zeta_1+sL] \times \cdots \times [\zeta_{d-1},\zeta_{d-1}+sL] \times [\zeta_{d},\zeta_d+tL]. \]
    These observations immediately reveal the following corollary.
    \begin{corollary}\label{cor:rect}
    Given $\eta,L,\sigma$, there exists $C>0$ such that for all $u \in L^2(\Rb^d)$ and $a:\Rb^d \ra [0,1]$ satisfying
        \begin{itemize}
            \item $\supp \hat u \subset \hat R$ for some $\hat R \in \mathsf R^{\sigma,\bot}_{\frac{1}{s},\frac{1}{t}}$,
            \item $\displaystyle \frac 1{|R|} \int_R a(x) \, dx \ge \eta$ for all $R$ in $\mathsf R^{L,\bot}_{s,t}$,
        \end{itemize}
        we have 
        \begin{equation}
            \ltwo{u} \leq C \ltwo{au}.
        \end{equation}
    \end{corollary}
    \begin{proof}
        For any $z \in \Rb^d$, there is a rectangle $R \in \mathsf R^{L,\bot}_{s,t}$ such that $D_{s,t} [z,z+L]^d=R$. Then with $\tilde a = a(D_{s,t} \cdot)$, using a change of variables and applying our assumption on $a$ we have 
        \begin{equation}
            \frac{1}{L^d} \int_{[z,z+L]^d} \tilde a (x) dx = (s^{d-1}t)^{-1} \frac{1}{L^d} \int_R a(x) dx \geq \eta.
        \end{equation}
        That is $\tilde a$ is $(L,\eta)$ relatively dense.

        On the other hand since $\supp \hat{u} \subset \hat{R}$ for some $\hat{R} \in \mathsf R_{\frac{1}{s},\frac{1}{t}}^{\sigma, \bot}$, then $\supp (\hat{u})_{\frac{1}{s},\frac{1}{t}} \subset [\zeta, \zeta+\sigma]^d$ for some $\zeta \in \Rb^d$. Therefore $\tilde a$ and $u_{s,t}$ satisfy the hypotheses of Theorem \ref{thm:pls} and so there exists $C>0$ such that  
        \begin{equation}
            \ltwo{u}=\ltwo{u_{s,t}} \leq C \ltwo{\tilde a u_{s,t}} = C\ltwo{au},
        \end{equation}
        where the equalities follow since $D_{s,t}$ preserves the $L^2$ norm.
    \end{proof}
\subsection{Almost Orthogonality Lemma}
    To extend the uncertainty principle of Corollary \ref{cor:rect} from functions with Fourier support in a rectangle, to functions with Fourier support in the annulus $A_{\beta,\lambda}$, we'll use the following almost orthogonality lemma. Its proof follows \cite[Theorem 8]{GJM22} very closely, but we state it in a slightly more general form.
\begin{lemma}\label{l:almost}
Fix $0<\delta<1$ and a dimension $d$. There exists $C_{\d, d}>0$, such that for any  $E,\Sigma \subset \mathbb R^d$, $\Xi = \{\xi_j\}_{j=1}^N$ a collection of $\frac 12$ separated points, and $\mathsf S$ a cover of $\Sigma$ such that
\begin{itemize}
    \item[(a)] There exists $C_1>0$ such that $\ltwo{w} \le C_1 \ltwo{\mathbbm{1}_E w}$ whenever $\supp \hat w \subset S \cap \Sigma$ for some $S \in \mathsf S$;
    \item[(b)] there exists $T\ge 1$ such that for each $j$, there exists $S_j \in \mathsf S$ satisfying
        \[ \dist(\xi_j,\Sigma \backslash S_j) \ge T,\]
\end{itemize}
    then we have 
        \begin{equation}\label{e:almost} \norm{u}_{L^2} \le C_{\d,d}(C_1+1) \left( \ltwo{\mathbbm{1}_{E_\delta} u} + \frac{1}{T} \ltwo{u} \right)\end{equation}
    for all $u \in L^2$ satisfying $\supp \hat u \subset \Sigma \cap \Xi_1=\Sigma \cap \left(\bigcup_{j=1}^N B(\xi_j,1)\right)$.
\end{lemma}
\begin{proof}
Let $\psi \in C_0^\infty( B(0,1) )$ with $\hat\psi \gtrsim 1$ on $B(0,1)$. Since $\delta$ is fixed, set $\varphi=\delta^{-d} \psi(\delta^{-1} \cdot)$. Then $\varphi$ is supported in $B(0,\delta)$ and, since $\delta <1$, $\hat \varphi \geq C_{\d}$ on $B(0,1/\d) \supset B(0,1)$.
Finally, define $\varphi_j$ as
    \[ \hat \varphi_j(\xi) = \hat \varphi( \xi - \xi_j).\]
Now setting $w_j = u * \varphi_j$, since $\supp \hat u \subset \Xi_1$ and $\hat \varphi_j \geq C_{\d}$ on $B(\xi_j,1)$, by Plancherel we have 
    \begin{equation} \label{eq:wjtriangle}
       \ltwo{u}^2 \leq C_{\d} \sum_{j=1}^N \ltwo{w_j}^2.
    \end{equation}
Further split $w_j = w_j^{(1)} + w_j^{(2)}$ where $w_j^{(1)}$ is the (rough) Fourier cut-off of $w_j$ to $S_j$
By assumption (a),
\begin{equation}\label{e:uncp-wj1} \ltwo{w_j^{(1)}} \leq C_1 \norm{\mathbbm{1}_E w_j^{(1)}}_{L^2}. 
\end{equation}
For $w_j^{(2)}$, by the smoothness of $\varphi_j$, $\hat \varphi_j$ decays in $\xi$, and we have
    \[ \norm{w_j^{(2)}}_{L^2}^2 \leq C_{\d} \int_{\xi \in \Sigma \backslash S_j } \frac{\abs{\hat u(\xi)}^2}{(1+|\xi-\xi_j|)^{d+1}} \, d\xi.\]
Using the triangle inequality twice along the lines of \eqref{e:triangle}, we combine the above two displays, to obtain
    \begin{align} 
        \ltwo{w_j} 
        &\leq C_1 \norm{\mathbbm{1}_E w_j}_{L^2} + (C_1+1)\left( C_{\d} \int_{\xi \in \Sigma \backslash S_j} \frac{\abs{\hat u(\xi)}^2}{(1+|\xi-\xi_j|)^{d+1}} \, d\xi \right)^{\frac 12}. \label{eq:wjintermed}
    \end{align}
We will sum the square of the above inequality over $j$ to achieve \eqref{e:almost}. To control the second term, we use (b). Indeed, for each $\xi \in \mathbb \supp \hat u$, if $\xi \in \Sigma \backslash S_j$, then $|\xi-\xi_j| \ge T$. Then using that the $\xi_j$ are $\frac{1}{2}$ separated to control the number of $\xi_j$ in the annuli $2^k T \leq |\xi_j-\xi| \leq 2^{k+1} T$ we have
    \begin{equation}\label{e:Rjc}\sum_{j=1}^N \frac{\mathbbm{1}_{\Sigma \backslash S_j}(\xi)}{(1+|\xi-\xi_j|)^{d+1}} \le \sum_{k=0}^\infty \sum_{\substack{j \\ 2^{k} T \le |\xi_j-\xi| \le 2^{k+1}T}} \frac{1}{(2^{k}T)^{d+1}} \le \sum_{k=0}^\infty \frac{(2^{k+1}T)^d}{(2^{k}T)^{d+1}} \leq \frac{2^{d+1}}{T}.\end{equation}
Thus 
\begin{equation}\label{eq:almostOrthogIntermed1}
    \sum_{j=1}^n \int_{\xi \in \Sigma \backslash S_j} \frac{|\hat{u}(\xi)|^2}{(1+|\xi-\xi_j|)^{d+1}} d \xi \leq \frac{2^{d+1}}{T} \|u\|_{L^2}^2.
\end{equation}
To estimate the $\mathbbm{1}_E w_j$ term, since $\varphi_j$ is supported in a ball of radius $\delta$, we have that $w_j=(u \mathbbm{1}_{E_\delta}) * \varphi_j$ on $E$. Therefore, repeating the computation in \eqref{e:Rjc} without the cutoff and with $T=1$
\begin{equation}
    \sum_{j=1}^N \frac{1}{(1+|\xi-\xi_j|)^{d+1}} \leq 2^d+ \sum_{k=0}^{\infty} \sum_{\substack{j \\ 2^{k}  \le |\xi_j-\xi| \le 2^{k+1}}} \frac{1}{2^{k(d+1)}}\leq 2^d + 2^{d+1} \leq  2^{d+2}.
\end{equation}
Therefore
    \begin{equation} \label{eq:almostOrthogIntermed2}
    \sum_{j=1}^N \norm{\mathbbm{1}_E w_j}_{L^2}^2 \le C_{\d} \sum_{j=1}^N \int_{\mathbb R^d} \frac{ \abs{\hat( u \mathbbm{1}_{E_\delta})(\xi)}^2 }{(1+|\xi-\xi_j|)^{d+1}} \, d\xi \leq C_{\d}2^{d+2} \norm{u \mathbbm{1}_{E_\delta}}_{L^2}^2.
    \end{equation}
    Applying \eqref{eq:almostOrthogIntermed1} and \eqref{eq:almostOrthogIntermed2} to \eqref{eq:wjtriangle} and \eqref{eq:wjintermed} we have 
    \begin{equation}
        \ltwo{u}^2 \leq C_{\d} \sum_{j=1}^n \|w_j\|_{L^2}^2 \leq C_{\d}(C_1+1)^2\left( 2^{d+2} \ltwo{u \mathbbm{1}_{E_\d}}^2 + \frac{2^{d+1}}{T} \ltwo{u}^2 \right).
    \end{equation}
    Taking square roots and defining $C_{\d,d}$ gives the desired inequality.
\end{proof}
Theorem \ref{thm:frac} will follow when Corollary \ref{cor:rect} and Lemma \ref{l:almost} are combined using some elementary geometry of the annulus $A_{\beta,\lambda}$.
\subsection{Proof of Sufficiency in Theorem \ref{thm:frac}}
    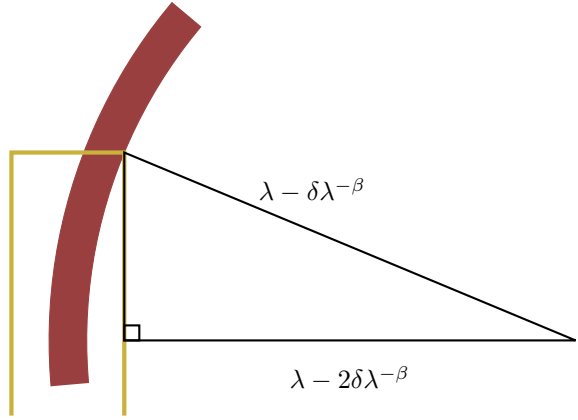
\begin{figure}[h]
        \centering
        \begin{tikzpicture}

            \draw [domain=140:185, name path=A, persred] plot ({7*cos(\x)}, {7*sin(\x)});
            \draw [domain=140:185, name path=B, persred] plot ({6.5*cos(\x)}, {6.5*sin(\x)});
            \tikzfillbetween[of=A and B]{persred};
            \draw[ultra thick, sand] (-7.5,-1) -- (-7.5,2.5) -- (-6,2.5) -- (-6,-1);
            \draw[thick] (-6,0) rectangle (-5.8,0.2);
            \node at (-3,-0.5) {$\lambda-2\d\lambda^{-\beta}$};
            \node at (-3.5,2) {$\lambda-\d \lambda^{-\beta}$};
            \draw[thick] (0,0) -- (-6,0) -- (-6,2.5) -- (0,0);
        \end{tikzpicture}
        \caption{A portion of the annulus $A_{\beta,\lambda}$ is pictured in red. The base of the triangle is $\lambda-2\d\lambda^{-\beta}$ and the hypotenuse is $\lambda-\d\lambda^{-\beta}$. Therefore the height of the yellow rectangle is $\sim \lambda^{\frac{1-\beta}{2}}$ (approaching $\infty$ for $\beta < 1$) and its width is $4\d \lambda^{-\beta}$. To verify (b), we note that if $\xi \in A_{\beta, \lambda}$ and outside the rectangle, and $\xi_j =(0,\ldots,0,-\lambda)$ is the center of the rectangle, then $|\xi-\xi_j| \gtrsim \lambda^{\frac{1-\beta}{2}}$. }
        \label{fig:annuluscover}
    \end{figure}
    Assume that $a_{\beta,L}>0$, and we will show there exists $C>0$ such that for all $u \in L^2$ with $\supp \hat{u} \subset A_{\beta,\lambda}$ for some $\lambda \geq 1$ we have 
    \begin{equation}
        \|u\|_{L^2} \leq C \|au\|_{L^2}.
    \end{equation}
    
    Consider the rectangles $\mathsf R^L_\beta$ defined at the start of Section \ref{s:fractional}. These are the rectangles in 
        \[ \mathsf R^{L,\bot}_{\lambda^\beta, \lambda^{\frac{\beta-1}{2}}}\]
    together with their rotations. So, with a view towards applying Corollary \ref{cor:rect}, define $\hat {\mathsf R}^{\sigma}_\beta$ to be the $\mathsf R^{\sigma,\bot}_{\lambda^{-\beta},\lambda^{\frac{1-\beta}{2}}}$ and their rotations; we will take $\hat{R}^{\sigma}_{\beta} = \mathsf{S}$ and $A_{\beta,\lambda}=\Sigma$.
Since $a_{\beta,L} \ge \eta$, by \eqref{e:switch} $(\mathbbm{1}_E)_{\beta,L} \ge \frac{\eta}{2}$ with $E = \{ a \ge \eta/2\}$. Thus by Corollary \ref{cor:rect}, Lemma \ref{l:almost} condition (a) is verified.

    To verify (b), we take $\{\xi_j\}$ to be a collection of $\frac 12$ separated points on the circle of radius $\lambda$. Figure \ref{fig:annuluscover} shows that the annulus $A_{\beta,\lambda}=\Sigma$ can be covered by rectangles $R_j$ with centers $\xi_j$ from $\hat{\mathsf{R}}^\sigma_{\beta}$ and that 
        \[ \dist(\xi_j,R_j^c \cap A_{\beta,\lambda}) \gtrsim \lambda^{\frac{1-\beta}{2}}.\]
Since $a$ is uniformly continuous, we can find $\delta>0$ such that $E_\delta \subset \{a \ge \frac{\eta}{4}\}$. With this $\delta$ by Lemma \ref{l:almost} for any $u \in L^2$ with $\supp \hat{u} \subset A_{\beta,\lambda}$ we have
\begin{equation}
    \ltwo{u} \leq C_{\d,d} (C_1+1) \left( \|\mathbbm{1}_{E_{\d}}u\|_{L^2} + \lambda^{\frac{\beta-1}{2}}\ltwo{u} \right),
\end{equation}
where neither $C_{\d,d}$ nor $C_1$ depends on $\lambda$. Then, since $\beta < 1$, for $\lambda_0$ large enough, for $\lambda \geq \lambda_0$ we can absorb the second term on the right hand side back into the left to obtain 
\begin{equation}
    \ltwo{u} \leq C \ltwo{au}.
\end{equation}

Now, for $\lambda$ small, we can appeal directly to the PLS Theorem (recall Theorem \ref{thm:pls} above). Indeed, for $1 \leq \lambda \le \lambda_0$, note that 
    \[ a_{\beta,L} \ge \eta\]
 implies that $a$ is relatively dense by restricting attention to rectangles in $\mathsf R^L_{\beta}$ with $\lambda=1$. Furthermore $A_{\beta, \lambda}$ can be placed inside a cube of side length $2(\lambda_0+\d)$, then by the standard PLS Theorem (Theorem \ref{thm:pls}) with $\sigma = 2(\lambda_0+\delta)$ there exists $C>0$ such that for all $u \in L^2$ with $\supp \hat{u} \subset A_{\beta,\lambda}$ we have 
 \begin{equation}
     \ltwo{u} \leq C \ltwo{au}.
 \end{equation}
Therefore we have the desired inequality for $1 \leq \lambda \leq \lambda_0$ and $\lambda \geq \lambda_0$.
 
\subsection{Proof of Necessity in Theorem \ref{thm:frac}}
Suppose there exists $C_1>0$ such that 
\begin{equation}
    \ltwo{u} \leq C_1 \ltwo{au} 
\end{equation}
for all $u \in L^2$ with $\supp \hat u \subset A_{\beta,\lambda}$ and we will show $a_{\beta,L} >0$. To do so we will use the fact that some rectangles from $\hat{\mathsf R}^{\sigma}_{\beta}$ fit inside of $A_{\beta, \lambda}$ for $\sigma$ small. In fact, let us take $\sigma = \min\{c_\beta,1\}\delta$ where $c_\beta$ is a small constant chosen according to Figure \ref{fig:annuluscontain}.
         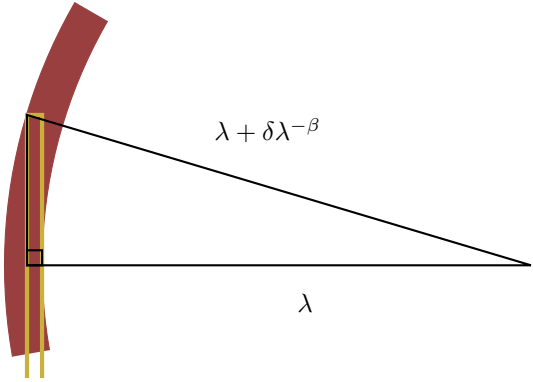
\begin{figure}
        \centering
        \begin{tikzpicture}
            \draw [domain=150:190, name path=A,persred] plot ({7*cos(\x)}, {7*sin(\x)});
            \draw [domain=150:190, name path=B, persred] plot ({6.5*cos(\x)}, {6.5*sin(\x)});
            \tikzfillbetween[of=A and B]{persred};
            \draw[ultra thick, sand] (-6.7,-1.5) -- (-6.7,2) -- (-6.5,2) -- (-6.5,-1.5);
            \draw[thick] (-6.7,0) rectangle (-6.5,0.2);
            \node at (-3,-0.5) {$\lambda$};
            \node at (-3.5,1.8) {$\lambda+\delta \lambda^{-\beta}$};
            \draw[thick] (0,0) -- (-6.7,0)--(-6.7,2) -- (0,0);
        \end{tikzpicture}
        \caption{The annulus $A_{\beta,\lambda}$ with an inscribed rectangle. The height of the rectangle is $c_\beta \delta \lambda^{\frac{1-\beta}{2}}$ and its width is $\d \lambda^{-\beta}$. The rectangle is centered at $(0,\ldots,0,-\lambda+\frac{\d}{2}\lambda^{-\beta})$}
        \label{fig:annuluscontain}
    \end{figure}

Now assume that $a_{\beta,L}=0$ and we will obtain a contradiction. Then for all $\ep>0$ there exists $R_{\ep} \in \mathsf{R}^L_{\beta}$ such that 
    \begin{equation}
        \frac{1}{|R_{\ep}|} \int_{R_{\ep}} a(x) dx < \ep.
    \end{equation}
    Fix a function $\phi \in C_0^{\infty}(B(0,\sigma))$ of unit $L^2$ norm, and let $\psi=\Fc^{-1}({\phi})$. We can find a large cube $Q$ of side length $\ell$ such that $\|\psi\|_{L^2(Q^c)} < \frac{1}{2C_1}$. Now fix an $R_{\ep}$ with $\ep$ satisfying 
    \begin{equation}
        \ep < \frac{1}{4C_1^2 \lp{\psi}{\infty}^2 |Q|}.
    \end{equation}
    Up to translating and rotating $a$ there exists $\lambda \geq 1$ such that $R_{\ep} \supset D_{s,t} Q$ with
    \begin{equation}
        s= \frac{L}{\ell} \lambda^{\frac{\beta-1}{2}}, \quad t=  \frac{L}{\ell}\lambda^{\beta}.
    \end{equation}
Therefore, applying a change of variables we obtain 
	\begin{equation}\label{e:Rloc} \int_{R_{\ep}^c} |\psi_{\frac 1s,\frac 1t}|^2 < \frac{1}{4C_1^2}.\end{equation}
Now $\phi_{s,t}$ is supported in the rectangle $D_{\frac 1s,\frac 1t} [-\sigma,\sigma]^d$. However, as pictured in Figure \ref{fig:annuluscontain}, the rectangle
    \[ \hat R_\ep \coloneq D_{\frac 1s,\frac 1t} [-\sigma,\sigma]^d - (0,\ldots,0,\mu) , \quad \mu = \lambda - \frac \delta 2 \lambda^{-\beta}\]
is contained in $A_{\lambda,\beta}$. Therefore, set
    \[ \tilde \psi(z) = e^{i\mu z} \psi_{\frac 1s , \frac 1t} (z),\]
so that $\mathcal F \tilde \psi(z) = \phi_{s,t}( z+ \mu)$ is supported on $\hat R_\ep \subset A_{\lambda,\beta}$, and $|\tilde \psi| = |\psi_{\frac 1s,\frac 1t}|$.
Then using the assumed uncertainty principle, the fact that $0 \le a \le 1$, and the decay of $\tilde \psi$ outside $R_\ep$ from \eqref{e:Rloc}, we have
	\[ \begin{aligned} 1= \norm{\tilde \psi}_{L^2} &\le C_1 \left( \ip{a \mathbbm 1_{R_\ep} \tilde \psi}{\tilde\psi}^{1/2} + \ip{a \mathbbm 1_{R_\ep^c} \tilde \psi}{\tilde\psi}^{1/2} \right)\\ 
	& \le C_1 \norm{\psi_{\frac{1}{s}.\frac{1}{t}}}_{L^\infty} \left( \int_{R_{\ep}} a \right)^{\frac 12} + \frac 12. \end{aligned}\]
Rearranging this is 
\begin{equation}
    \frac{1}{2C_1 \lp{\psi_{\frac{1}{s},\frac{1}{t}}}{\infty}} \leq \left( \int_{R_{\ep}} a \right)^{1/2}
\end{equation}
Recalling the definition of $\psi_{\frac 1s,\frac 1t}$, we have $\norm{\tilde \psi}_{L^\infty} = \norm{\psi_{\frac 1s,\frac 1t}}_\infty = \norm{\psi}_\infty (s^{d-1}t)^{-\frac 12}$. But
	\[ (s^{d-1}t)^{\frac 12}= \left( \frac{|R_{\ep}|}{|Q|} \right)^{\frac 12},\]
so we obtain a lower bound 
	\[ \frac{1}{|R_{\ep}|} \int_{R_{\ep}} a \ge \frac{1}{4C_1^2 \norm{\psi}_\infty^2 |Q|} > \ep\]
This is a contradiction, and so we must have $a_{\beta,L}>0$.

    \section{The endpoint case: Schr\"odinger observability}

In this section, we develop the propagation of singularities technique on $\mathbb R^d$ which allows us to push slightly beyond GCC, to what we term the \textit{comb GCC}. Rather than localizing the frequency to rectangles and covering the annulus with such rectangles, we localize to annular caps of width $\ep_0\lambda$, for some $\ep_0$ small. 

\subsection{Comb GCC}\label{s:combGCC2}

We first state our definitions and examples in dimension $d=2$ for clarity, but in Section \ref{s:higher} below, we will outline the extension to higher dimensions. And the proofs in Section \ref{s:comb-proof} will be in any dimension.
\begin{definition}\label{d:comb}
    Given a unit vector $\theta = (\cos \vartheta,\sin \vartheta) \subset \mathbb S^1$, we define $\theta a$ to be the counterclockwise rotation of $a$ by $\vartheta$. More precisely, we also use $\theta x$ to denote the rotation of a vector $x \in \mathbb R^2$ which sends the unit vector $(0,1)$ to $\theta$. Then, for functions $a: \mathbb R^2 \to [0,1]$, we use $\theta a$ to denote the function $a \circ \theta$. 
     Then we define
     \[ a_{\theta,M} (x) = \inf_{t \in \mathbb R} \frac{1}{M}\int_{t}^{t+M} \theta a( x ,y) \, dy, \quad x \in \mathbb R.\]
\end{definition}
    In words, this captures the lower density of $a$ at scale $M$ along the line passing through $x$ parallel to $\theta$.
Note, $a_{\theta,M}(x)$ is uniformly bounded below over all $x$, if and only if $a$ satisfies the GCC in direction $\theta$. Our comb GCC weakens this requirement to assume only that $a_{\theta,M}$ itself satisfies the GCC (or equivalently relative density) as a function on $\mathbb R$. 

\begin{definition}\label{def:combGCCdird=2}
    We say $a$ satisfies the $(M,L,\eta)$ comb GCC in direction $\theta$ if $a_M(\cdot,\theta)$ is $(L,\eta)$ relatively dense, i.e.
    \[ \inf_{x \in \mathbb R} \frac{1}{L} \int_{x}^{x+L} a_{\theta,M}(x) \, dx \ge \eta.\]
\end{definition} 
It is indeed the case that if $a$ satisfies the $(M,L,\eta)$ comb GCC in every direction, then the Schr\"odinger equation is long-time observable from $a$.  In practice however, this condition is much too strong, so we now introduce a quantitative relaxation which we call the comb GCC.

Our idea with be to use what we call an effective covering of the sphere. 

\begin{definition} Fix $L,\eta>0$. For each $\rho>0$ and $\lambda \ge 1$ we say $a$ has a $(\rho,\lambda)$ effective covering if there exists a set of directions $\Theta \subset \mathbb S^{1}$, and for each $\theta \in \Theta$ an arc width $\ep_{\theta}$ and length $M_{\theta}$,
such that 
\begin{itemize}
    \item $\displaystyle \mathbb S^1 \subset \bigcup_{\theta \in \Theta} \arc_\theta, \qquad  \text{Arc}_\theta := \left\{ \phi \in \mathbb S^{1} : \abs{ \phi - \theta } \le \ep_\theta\right\},$
    \item $a$ satisfies the $(M_\theta,L,\eta)$ comb GCC in direction $\theta$ with
    \begin{equation}\label{e:comb-eff}\ep_\theta M_\theta + (\ep_\theta \lambda)^{-1} < \rho.\end{equation}
\end{itemize}
\end{definition}
\begin{remark}
Note that the condition \eqref{e:comb-eff} actually encodes two separate conditions. First, we wish to guarantee that $M_{\theta}(\ep_{\theta} +\lambda^{-1})$ is sufficiently small, which is used to prove the high frequency resolvent estimate for functions frequency localized in $\text{Arc}_\theta$ in Proposition \ref{p:comb-uncp}. Second, we must impose $\ep_\theta \lambda$ is large to combine functions frequency localized to $\text{Arc}_\theta$, using the almost orthogonality Lemma \ref{l:almost}. There, $T \sim \ep_\theta \lambda$ must be large to absorb the error as in Theorem \ref{thm:frac}.

In checking the first condition, is it useful to note that
    \[ M_\theta(\ep_\theta + \lambda^{-1}) = \ep_\theta M_\theta + \frac{\ep_\theta M}{\ep_\theta \lambda} < \rho + \rho^2. \]
\end{remark}
\begin{remark}\label{r:rel-dense}
Note that if $a$ satisfies the comb GCC in a single direction, then by Fubini's theorem $a$ must be  relatively dense; see the proof of Lemma \ref{l:rd-comb} below.
\end{remark}
Finally, we are ready to define comb GCC.
\begin{definition}\label{def:combGCCd=2}
    We say $a$ satisfies the comb GCC if for each $\rho>0$ there exists $\lambda_0 \ge 1$ such that for each $\lambda \ge \lambda_0$, $a$ has a $(\rho,\lambda)$ effective covering.
\end{definition}

\begin{theorem}\label{thm:comb}
If $a$ is a non-negative uniformly continuous function satisfying the comb GCC, then the Schr\"odinger equation is observable by $a$.
\end{theorem}

The main ingredient is the following uncertainty principle for functions supported in the Schr\"odinger annulus (henceforth we drop the subscript $\beta$ as $\beta$ will always be $1$ now)
	\[ A_\lambda \coloneq A_{\lambda,1} = \left\{ \xi \in \mathbb R^2 : \lambda - \delta \lambda^{-1} \le |\xi| \le \lambda + \delta \lambda^{-1} \right\},\]
	and a sector of width $\ep_0$, which we denote by
	\[ S_{\theta,\ep_0} = \left\{ \xi \in \mathbb R^2 : \left| \frac{\xi}{|\xi|} - \theta \right| < \ep_0 \right\}. \]
Its proof uses a microlocal propagation of singularities approach, and is proved in Section \ref{s:comb-proof} below.
\begin{proposition}\label{p:comb-uncp}
Let $L,\eta>0$ and $\omega$ be a modulus of continuity. There exists $C,\delta>0$ such that
\begin{itemize}
	\item for all $\lambda \ge 1$ and $M,\ep_0>0$ satisfying 
	\begin{equation}\label{e:Meph}  M(\ep_0 + \lambda^{-1}) < \frac{1}{C},\end{equation}
	\item all $u$ with $\supp \hat u$ contained in $A_{\lambda} \cap S_{\theta,\ep_0}$
	\item all $a$ which are $\omega$-continuous and satisfy the $(M,L,\eta)$ comb GCC in direction $\theta$,
\end{itemize}
there exists $X \subset \mathbb R^2$ such that
	\begin{equation}\label{e:comb-uncp} \norm{u}_{L^2(\mathbb R^2)} \le C\norm{u}_{L^2(X)}, \quad X_\delta \subset \{ x \in \mathbb R^2 : a(x) \ge \delta\}. \end{equation}
\end{proposition}
Assuming this proposition we now give the proof of Theorem \ref{thm:comb} in dimension 2. The idea of the proof is to combine together the sectors from Proposition \ref{p:comb-uncp} using Lemma \ref{l:almost} to obtain an uncertainty principle like \eqref{e:uncp}. We can then apply Lemma \ref{l:equiv} to obtain a high frequency observability resolvent estimate. Then using the standard tools of Lemma \ref{l:lowFreqResolve} and Theorem \ref{thm:Miller} we convert the observability resolvent estimate into an observability estimate. 

\begin{proof}[Proof of Theorem \ref{thm:comb} in dimension 2]Let $C,\delta>0$ be the constants provided by Proposition \ref{p:comb-uncp} and, for this $\delta$, let $C_{\delta,2}$ be the constant provided by Lemma \ref{l:almost}. Let now $0<\rho<\frac{1}{6 C_{\d,2}(C+1)}$. Since $a$ satisfies the comb GCC, there exists $\lambda_0 \geq 1$ such that for all $\lambda \geq \lambda_0$, $a$ has a $(\rho, \lambda)$ effective covering $(\Theta, \{\ep_{\theta}, M_{\theta}\}$). 
     
For each $(\rho,\lambda)$ effective covering, we seek to apply Lemma \ref{l:almost}, with $\Sigma=A_{\lambda}$, $E=X$, and $\mathsf{S}=\{A_{\lambda} \cap S_{\theta, 2\ep_{\theta}}: \theta \in \Theta\}$. After potentially taking $\rho$ smaller and $\lambda_0$ larger, the conclusion of Proposition \ref{p:comb-uncp} verifies assumption a) of Lemma \ref{l:almost}. On the other hand, to verify the separation condition b), we specify $\Xi$ to be a collection of $1/2$ separated points with gaps no greater than $1$ on the circle $\lambda\mathbb S^1$. Due to the covering property of $\{\text{Arc}_\theta : \theta \in \Theta\}$, each $\xi \in \Xi$ belongs to some $S_{\theta,\ep_\theta}$ and
    \begin{equation}
        \text{dist}(\xi, A_{\lambda} \backslash (A_{\lambda} \cap S_{\theta,2\ep_{\theta}}) ) \geq \frac 1 3 \ep_{\theta} \lambda >\frac{1}{3\rho}.
    \end{equation}
Therefore by Lemma \ref{l:almost}, for all $u \in L^2$ with $\supp \hat{u} \subset A_{\lambda} \cap \Xi_1 = A_{\lambda}$ we have 
    \begin{equation}
        \|u\|_{L^2} \leq C_{\d,2}(C+1) \left(  \|\mathbbm{1}_{X_{\d}} u\|_{L^2} + 3\rho \|u\|_{L^2} \right). 
    \end{equation}
    By our assumption that $\rho<\frac{1}{6C_{\d,2}(C_1+1)}$ we may absorb the error term from the right hand side back into the left, and thereby see that for all $\lambda \geq \lambda_0$ and all $u$ with $\supp \hat{u} \subset A_{\lambda}$
    \begin{equation}
        \ltwo{u} \leq \frac{C_{\delta,2}(C+1)}{2} \ltwo{au}.
    \end{equation}
    Now we apply Lemma \ref{l:equiv} with $\beta=1$ and $\gamma=2$ to obtain the following high frequency resolvent estimate: there exists $C>0$ such that for all $u \in H^2$, and $\lambda \geq \lambda_0^2$ we have 
    \begin{equation}
        \|u\|_{L^2}^2 \leq C \ltwo{au}^2 + C\ltwo{(-\Delta-\lambda) u}^2.
    \end{equation}
    Then by Lemma \ref{l:lowFreqResolve}, since $a$ is relatively dense (Remark \ref{r:rel-dense}), there exists $C>0$ such that for all $\lambda \in \Rb$ and $u \in H^2$ we have 
    \begin{equation}
    \|u\|_{L^2}^2 \leq C \ltwo{au}^2 + C\ltwo{(-\Delta-\lambda) u}^2.
    \end{equation}
    Finally, by Theorem \ref{thm:Miller} the Schr\"odinger equation is observable from $a$.
\end{proof}

\subsection{Product-Type Model}\label{s:product}
Our first application of the comb GCC condition will be concerning product-type observation functions.
\begin{proposition}\label{p:product}
If $f_j:\mathbb R \to [0,1]$, $j=1,2$ are uniformly continuous, $(M,\delta_j)$ relatively dense functions with 
	\[ \delta_1 + \delta_2 > 1,\]
then $a(x,y)=f_1(x) f_2(y)$ satisfies the comb GCC.
\end{proposition}

Such an $a$ immediately satisfies the $(M,M,\delta_1\delta_2)$ comb GCC in the horizontal and vertical directions. To see this in the vertical direction, $\theta=(0,1)$, note that because $g$ is $(\delta_2,M)$ relatively dense,
	\begin{equation}\label{e:prod-comb} a_{\theta,M}(x)\ge f_1(x) \delta_2.\end{equation}
Therefore $a_M$ is $(M,\delta_1 \delta_2)$ relatively dense. 

To handle the other directions, we first point out that for directions $\ep$ away from the horizontal or vertical directions, the GCC is satisfied for some uniform time that depends on $\d_1+\d_2-1$ and $\ep>0$.

\begin{lemma}\label{l:productGCC}
    Suppose $E$ is $(\d_1,1)$ relatively dense and $F$ is $(\d_2,1)$ relatively dense with $\d_1+\d_2>1$ and $0 < \ep <\frac 14$. Then, there exists $L,\eta >0$ such that for all $\theta=(s,t) \in \mathbb S^1$ with $\min\{|s|,|t|\} > \ep$, $\mathbf 1_{E \times F}$ satisfies the $(L,\eta)$ GCC in direction $\theta$.

\end{lemma}
\begin{proof}
We will only prove the case $s,t>0$, as the other cases follow by symmetry.
    Consider a line $\ell$ of length $L$ parallel to $\theta$ starting from the point $z_0=(x_0,y_0)$.
    Let $m=\frac{t}{s}$ be the slope of the line and $X_0=Ls$ be the $x$ displacement along its length.
    Then since $E$ is $(\d_1,1)$ relatively dense 
    \begin{equation}
        \abs{ (x_0,x_0+X_0) \cap E} \geq \d_1 \lfloor X_0 \rfloor .
    \end{equation}
    Now consider the subset of $\ell$ which project in the $x$ coordinate to $(x_0,x_0+X_0)\cap E$, that is let 
    \begin{equation}
        \ell(E)=\{(x,y) \in \ell: x \in E\}.
    \end{equation}
    Now if we project this set to $y$, 
    \begin{equation}
        \pi_y \ell(E)=\{y \in (y_0, y_0+mX_0): \exists x \in (x_0,x_0+X_0) \cap E, (x,y) \in \ell\}
    \end{equation}
    we have 
    \begin{equation}\label{eq:deltaGCCEreldense}
        \abs{ \pi_y \ell(E)} \geq \d_1 \lfloor X_0 \rfloor m.
    \end{equation}
    On the other hand by the $(\d_2,1)$ relative density of $F$ we have
    \begin{equation}\label{eq:deltaGCCFreldense}
        \abs{(y_0,y_0+mX_0) \cap F} \geq \d_2 \lfloor mX_0 \rfloor.
    \end{equation}
    Now we claim that 
    \begin{equation}
        \abs{\pi_y \ell(E)} + \abs{(y_0, y_0+mX_0) \cap F} \geq \abs{(y_0, y_0+ mX_0)} + \frac{\d_1+\d_2-1}{2} mX_0, 
    \end{equation}
    or equivalently by \eqref{eq:deltaGCCEreldense} and \eqref{eq:deltaGCCFreldense} it would be enough to show
    \begin{equation}\label{eq:deltaGCCclaim}
        \d_1 \lfloor X_0\rfloor m + \d_2 \lfloor m X_0\rfloor \geq mX_0 + \frac{\d_1+\d_2-1}{2} m X_0.
    \end{equation}
    If this claim is true, then by the inclusion-exclusion principle 
    \begin{equation}
        \abs{F \cap \pi_y \ell(E)} \geq \frac{\d_1+\d_2-1}{2} mX_0.
    \end{equation}
    The subset of $\ell$ which projects to this set in $y$ is exactly $\ell \cap (E\times F)$ so
    	\[ \abs{\ell \cap (E \times F)} = \abs{ \pi_y^{-1} (F \cap \pi_y \ell(E))} = \frac{L}{mX_0} \abs{ F \cap \pi_y \ell(E)} \ge 
    \frac{\d_1+\d_2-1}{2} L.\] 
    Therefore the $(L,\eta)$ GCC is established with $\eta = \frac{\delta_1+\delta_2-1}{2}$. It remains to choose $L$ large enough that \eqref{eq:deltaGCCclaim} holds. To this end, we divide by $mX_0$ and write $X_0=Ls$ and $mX_0=Lt$, to reduce to choosing $L$ such that
    	\[ \delta_1 \left( \frac{Ls-1}{Ls} \right) + \delta_2 \left( \frac{Lt-1}{Lt} \right) \ge \frac{\delta_1+\delta_2+1}{2}.\]
Indeed, letting $\alpha = \frac{\delta_1+\delta_2+1}{2(\delta_1+\delta_2)}$, the condition $\delta_1+\delta_2>1$ forces $\alpha<1$. Therefore, if $L \ge \frac{1}{(1-\alpha)\ep}$ then $Ls$ and $Lt$ are both $\ge \frac{1}{1-\alpha}$ and hence
	\[ \delta_1 \left( \frac{Ls-1}{Ls} \right) + \delta_2 \left( \frac{Lt-1}{Lt} \right) \ge \left( \delta_1+\delta_2\right )\alpha = \frac{\delta_1+\delta_2+1}{2}.\]
And the claim \eqref{eq:deltaGCCclaim} is proved, as desired.
\end{proof}
We now show that $a$ satisfies the comb GCC
\begin{proof}
    [Proof of Proposition \ref{p:product}]
    For any $\rho>0$, we will produce $\lambda_0 \geq 1$ such that for all $\lambda \geq \lambda_0$, $a$ has a $(\rho,\lambda)$ effective covering. 

    Let $\ep_1=\frac{\rho}{4M}$, then using \eqref{e:switch} to reduce $f,g$ to appropriate level sets, let $(L,\eta)$ be the GCC constants from Lemma \ref{l:productGCC} for directions $\theta$ at least $\frac{\ep_1}{2}$ away from horizontal or vertical directions. Let $\ep_2=\frac{\rho}{4L}$ and $\lambda_0=\frac{8}{\rho^2}\max(L,M)$. Then for all $\lambda \geq \lambda_0$ define 
    \begin{align}
       &\Theta = \Theta_1 \cup \Theta_2, \\
       & \Theta_1 = \left\{ (0,1), (1,0), (0,-1), (-1,0) \right\} \\
       &\Theta_2 = \left\{ (x,y) \in \mathbb S^1 : \min\{|x|,|y|\} \ge \ep_1/2 \right\}.
    \end{align}
    For $\theta \in \Theta_1$ we take $M_{\theta}=M$, $\ep_{\theta}=\ep_1$, noting that $a$ satisfies the $(M_{\theta}, M, \d_1\d_2)$ comb GCC in these directions as discussed in \eqref{e:prod-comb} above. For the remaining $\theta \in \Theta_2$ we take $M_{\theta}=L$, and $\ep_{\theta}=\ep_2$, noting that by Lemma \ref{l:productGCC} $a$ satisfies the $(M_{\theta},M,\eta)$ comb GCC in these directions. It is immediate that $\mathbb{S}^1 \subset \bigcup_{\theta \in \Theta} \text{Arc}_{\theta}$ and for all $\theta \in \Theta$
    \begin{equation}
        \ep_{\theta}  M_{\theta} + (\ep_{\theta} \lambda)^{-1} \leq \frac{\rho}{4}+ \frac{\rho}{2} \leq \rho.
    \end{equation}
    So this $\Theta$ is the desired $(\rho,\lambda)$ effective covering and $a$ satisfies the comb GCC.
\end{proof}

\subsection{Periodic Case}\label{s:periodic}
For periodic functions if we consider directions with rational slopes, either the GCC or the comb GCC is satisfied. 

\begin{proposition}\label{p:period}
    Let $a: \Rb^2 \ra [0,1]$ be continuous, non-trivial, and periodic. Then $a$ satisfies the comb GCC.
    
    In particular, there exists $\eta,T_0 > 0$ such that 
    the following two statements hold for all $\theta=(P,Q)/T_\theta$ where
        \[ P, Q \in \mathbb Z, \quad \gcd(P,Q)=1, \quad T = \sqrt{P^2+Q^2}.\]
    \begin{enumerate}
        \item If $T \geq T_0$, then $a$ satisfies the $(T+2,\eta)$ GCC in direction $\theta$. 
        \item If $T < T_0$ then $a$ satisfies the $(2T+4,1/T,\eta)$ comb GCC in direction $\theta$.
    \end{enumerate}
 \end{proposition}

\begin{proof}[Proof of Proposition \ref{p:period}]
We fix $\theta$ as above and may assume $a$ is the indicator of a $1$-periodic square of side length $\delta>0$ oriented along $\theta$. First note that if $T=1$, then either $\theta=(1,0)$ or $(0,1)$. In this case, $a$ immediately satisfies the $(T,T,\d^2)$ comb GCC in direction $\theta$; recall again the product-type argument \eqref{e:prod-comb}. Because of this, from now on we assume $T >1$, in particular both $|P|,|Q| \geq 1$.

Consider the square $\mathcal{S}$ of side length $\frac{|P|+|Q|}{T} \leq 2$ with sides parallel to $\theta$ and $\theta^{\bot}$, with bottom centered at a point $z \in \Rb^2$; see Figure \ref{fig:periodic}. This contains a $1 \times 1$ square, and up to replacing $\d$ by $\frac{\d}{C}$ for some dimensional constant $C$, the $ 1 \times 1$ square contains one of the $\delta$ side length squares where $a=1$, which we will call $S$. Let $\ell$ denote the bottom of $S$, which is a line of length $\d$ parallel to $\theta^{\bot}$.

1) Now for $\theta$ with $T \geq \d^{-1}$, for any $z \in \Rb^2$ consider the trajectory 
    \begin{equation}\label{e:traj} 
    z + s\theta, \quad 0 \le s \le T+2.
    \end{equation}
We claim that each such trajectory \eqref{e:traj} hits $\{ \ell+m : m \in \mathbb Z^2\}$ at least $\lfloor \delta T\rfloor \ge 1$ times.  
Let us denote such points by $\Sigma$. That will indeed prove part 1) since  
	\[ \int_0^{T + 2} a(z+s\theta) \, ds \ge \sum_{\xi \in \Sigma} \int_{0}^\delta a( \xi + s\theta) \, ds  \ge \delta \lfloor\delta T \rfloor \ge \frac{\delta^2T}{2}.\]
So now to prove our claim, consider all $\zeta \in \ell$ such that the $\theta^{\bot}$ components of $z$ and $\zeta$ differ by $\frac{n}{T}$ for some $n \in \mathbb{Z}$. Since $\ell$ has length $\delta$ there are at least $\lfloor \delta T\rfloor \geq 1$ such $\zeta$. Then by the construction of the large square $\mathcal{S}$ we have 
\begin{equation}
    \left|z+ \frac{n}{T} \theta^{\bot} - z\right| \leq \frac{|P| + |Q|}{2T}.
\end{equation}
Therefore $|n| \leq \frac{|P|+|Q|}{2}$. Now since $|P|, |Q|\geq 1$ we have $|n| \leq |P||Q|$.
\begin{figure}
\tikzmath{\z1 = 0.75; \z2 =0.25; \pp=2; \qq=5; \TT={sqrt(\pp^2+\qq^2)};\del=0.34;\z3=0.56;\z4=0.5;} 
\begin{tikzpicture}[scale=2]
        \draw[rotate around={338:(1,0)}] (0,0)--(0,2)--(2,2)--(2,0)--(0,0);
        \draw (0.6,0.3)--(1.6,0.3)--(1.6,1.3)--(0.6,1.3)--(0.6,0.3);
        \filldraw[rotate around={338:(1,0)}] (1,0) circle (1pt) node[below] {$z$};
        \draw[rotate around={338:(1,0)}] (1,0) -- (1,3);
        \draw[->,rotate around={338:(1,0)}] (1,4) -- (1,5);
        \filldraw[sand, rotate around={338:(1,0)}] (1,4.5) circle (1pt) node[right] {$\zeta+m$};
        \draw[->, rotate around={338:(1,0)}] (-.2,-.2)--(.5,-.2) node[right] {$\theta^{\bot}$};
        \draw[->, rotate around={338:(1,0)}] (-.2,-.2)--(-.2,.5) node[left] {$\theta$}; 
        \draw[imayou, pattern=north east lines, pattern color=imayou, rotate around={338:(1,0)}] (\z3,\z4) rectangle (\z3+\del,\z4+\del);
        \filldraw[sand, rotate around={338:(1,0)}] (\z3+\del/2,\z4) circle (1pt) node[left] {$\zeta$};
        \draw[dotted, rotate around={338:(1,0)}] (\z3+\del/2,\z4) -- (-.7,3.8) -- (1,4.5);
    \end{tikzpicture}
    \caption{The shaded region is $S$. The edge of $S$ passing through $\zeta$ is $\ell$.}
    \label{fig:periodic}
\end{figure}
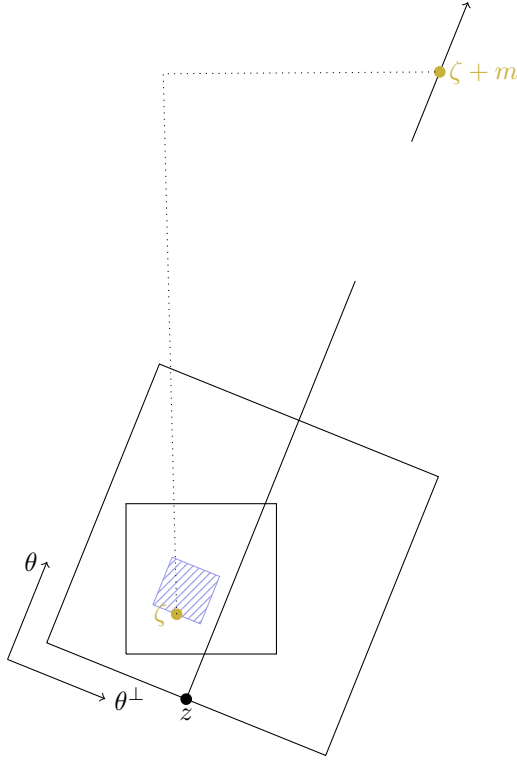
     Thus we can write 
     \begin{equation} \label{eq:periodicGCC1}
         \zeta = z + \frac{n}{T}\theta^\perp + s^* \theta, \quad |s^*| \le 2.
     \end{equation}
We will now show that $\zeta+m$ lies on the trajectory \eqref{e:traj} for some $m \in \mathbb Z^2$. Let $\mathfrak{a},\mathfrak{b}$ be integers such that $\mathfrak{a}P+\mathfrak{b}Q=n$. Since $|n| \leq |PQ|$, 
by Lemma \ref{l:Bezout} below, these $\mathfrak{a},\mathfrak{b}$ can be chosen such that $|\mathfrak{a}| \le |Q|, \quad |\mathfrak{b}| \le |P|$.
Then, one can verify
    \[ \frac{n}{T}\theta^{\bot} = \frac{n}{T^2} \begin{pmatrix} -Q \\ P \end{pmatrix} = \frac{\mathfrak{b}P-\mathfrak{a}Q}{T^2} \begin{pmatrix} P \\ Q \end{pmatrix} + \begin{pmatrix} -\mathfrak{b} \\ \mathfrak{a} \end{pmatrix} = \frac{\mathfrak{b}P-\mathfrak{a}Q}{T}\theta -m. \] 
Plugging this into \eqref{eq:periodicGCC1}, we have $\zeta = z + \left( \frac{\mathfrak{b}P-\mathfrak{a}Q}{T}+s^* \right) \theta - m$ with $m = \begin{pmatrix} \mathfrak{b} \\ -\mathfrak{a} \end{pmatrix} \in \mathbb Z^2$. Since $|\mathfrak{a}| \leq |Q|, |\mathfrak{b}|\leq |P|$ and $|s^*| \leq 2$
\begin{equation}
    \left| \frac{\mathfrak{b}P-\mathfrak{a}Q}{T}+s^* \right| \leq \frac{T^2}{T} + 2 \leq T+ 2,
\end{equation}
so indeed the trajectory \eqref{e:traj} hits $\{l+m:m\in \mathbb{Z}^2\}$ at least $\lfloor \d T \rfloor$ times. 

2) The second statement concerning the comb GCC is proved in a similar way.  
For $\theta$ with $T < \frac{1}{\d}$, for a point $x_0 \in \Rb$, let $z= \theta(x_0,0)$. Then the line segment 
    \[ \ell^* = \theta \left( \left[x_0-\frac{1}{2T},x_0+\frac{1}{2T}\right] \times \{0\} \right)\]
is a line segment of length $1/T>\delta$ parallel to $\theta^{\bot}$ and contained in the bottom of $\mathcal{S}$.

For each $\zeta \in \ell$ there exists $z^* \in \ell^*$ such that the $\theta^\perp$ components of $z^*$ and $\zeta$ differ by $\frac nT$ for some $n \in \mathbb Z$. Such $z^*$ exists due to the fact that the width of $\ell^*$ is exactly $1/T$. Moreover each $z^*$ is associated to only one $\zeta$ in this way, since $\delta < 1/T$. Now, using the construction of $\mathcal{S}$ and $l^*$ we have 
\begin{equation}
    \left|z^*+\frac{n}{T}\theta^{\bot}-z^*\right| \leq \frac{|P|+|Q|}{2T} + \frac{1}{2T}.
\end{equation}
Thus $|n| \leq \frac{|P|+|Q|+1}{2}$ and since $n \in \mathbb{Z}$, $|n| \leq |P||Q|$. Now, as shown above we can find $m \in \mathbb Z^2$ such that
    \[ \zeta +m = z^* + \theta s, \quad 0 \le s \le T+ 2.\]
Therefore, letting $Z \subset \ell^*$ be the collection of such $z^*$, we have $|Z|=\delta$ and for each $z^* \in Z$,
    \[ \int_0^{T + 2} a(z^*+\theta s) \, ds \ge \delta.\]
Now since $s \mapsto a(z^*+\theta s)$ is a $T$-periodic function on $\mathbb R$, is it also $(2T+4,\frac{\delta^2}{2T+4})$ relatively dense.
To pull this back to $a_{\theta,T}$, take $E= \pi_1(\theta^{-1} Z)$ to be the rotation of $Z$ back to the real line, so that for each $x \in E \subset [x_0-\frac{1}{2T},x_0+\frac{1}{2T}]$
    \[ a_{\theta,2T+4}(x) \ge \frac{\delta^2}{2T+4}.\]
But, since $|E|=\delta$ and $x_0$ was arbitrary, we obtain that $a_{\theta,2T+4}$ is $(1/T,\frac{T\delta^2}{2T+4})$ relatively dense, and therefore $a$ satisfies the comb GCC.

3) Now we will use 1) and 2) to generate an effective covering of $\mathbb S^1$, thereby showing that $a$ satisfies the comb GCC. Given $\rho>0$ let $\lambda_0 = (\frac{20}{\rho})^4$ and for $\lambda \geq \lambda_0$ we will define a $(\rho,\lambda)$ effective covering.

We cover the sphere $\mathbb S^1$ with neighborhoods of the following angles
	\[ \Theta=\left\{\theta = \frac{(P,Q)}{T}: \quad \gcd(P,Q)=1, \quad 1 \le T \le 2\lambda^{\gamma}\right\},\]
for some $\gamma >0$ small. Any $0<\gamma < \frac 12$ will work in what follows, so we choose $\gamma= \frac 14$.
Then consider the arcs
	\begin{equation}\label{e:neighborhood}\arc_\theta =\left\{ \phi \in \mathbb S^1 :  \abs{\phi-\theta} \le \frac{2}{\lambda^{\gamma}T}, \right\}.\end{equation}
Let us demonstrate that this is indeed an effective covering of the sphere $\mathbb S^1$. Consider $\phi=(x,y) \in \mathbb S^1$. Up to permuting $x$ and $y$, we may assume $|x| \geq |y|$. Then, by Dirichlet's theorem we can find $1 \le P \le \lambda^{\gamma}$ and $Q \in \mathbb Z$ such that
	\[ \left| \frac{y}{x}-\frac{Q}{P} \right| \le \frac{1}{P\lambda^{\gamma}}, \quad \gcd(P,Q) =1.\]
In general, the pair $(Q,P)$ does not satisfy $\gcd(Q,P) =1$, but if we do simplify it, we only decrease the denominator, so the above display still holds. Notice also that such $(P,Q)$ must satisfy $P \ge|Q|$, 
forcing $T \le \sqrt{2}P \le 2P$. Setting $\theta = \frac{(P,Q)}{T}$, then $\theta \in \Theta$ and  applying the law of cosines we have 
	\[ |\phi-\theta| \leq \left| \frac{y}{x} - \frac{Q}{P} \right| \leq \frac{1}{P\lambda^{\gamma}} \leq \frac{2}{T \lambda^{\gamma}}.\]
Therefore, we have shown the existence of a $\theta$ such that $\phi \in \arc_\theta$. 
Thus $\bigcup \arc_\theta$ covers the sphere $\mathbb S^1$. We choose, $\ep_\theta = \frac{4}{\lambda^\gamma T}$. 

For each $\theta$ we either have $T_{\theta} \geq T_0$, so $a$ satisfies the $(T+2,\eta)=(M_{\theta},\eta)$ GCC in direction $\theta$, or $T_{\theta} \leq T_0$, so $a$ satisfies the $(2T+4,\frac{1}{T},\eta) = (M_{\theta}, L,\eta)$ comb GCC in direction $\theta$.
So $M_{\theta} \leq 6T$ and by direct computation we see
\begin{equation}
    M_{\theta} \ep_{\theta} + (\ep_{\theta} \lambda)^{-1} \leq \frac{12}{\lambda^{\gamma}} + \frac{T}{4\lambda^{1-\gamma}} \leq \frac{12}{\lambda^{\gamma}} + \frac{1}{2 \lambda^{1-2\gamma}} < \rho
\end{equation}
where the last inequality follows when $\gamma=\frac{1}{4}$ and since $\lambda \geq \lambda_0=(\frac{20}{\rho})^4$. Therefore $\Theta$ is a $(\rho,\lambda)$ effective covering, so $a$ satisfies the comb GCC.
\end{proof}

\begin{lemma}\label{l:Bezout}
Let $P,Q \in \mathbb Z$ such that $\gcd(P,Q) = 1$. Then, for each 
    \[ n \in \left\{ 1,2,\cdots, |PQ| \right\},\]
there exists $a,b \in \mathbb Z$ such that
    \[ aP + bQ =n, \quad |a| \le |Q|, \quad |b| \le |P|.\]
\end{lemma}
\begin{proof}
    We can take both $P$ and $Q$ to be positive. Bezout's Theorem guarantees integers $a,b$ such that $aP + bQ=1$. Since either $a$ or $b$ must be positive, we may assume $a > 0$. Therefore,
    \begin{equation}\label{eq:bezoutOne}
        (na-kQ)P + (nb+kP)Q = n
    \end{equation}    
    for any integer $k$. Choose $k = \lceil \frac {na}{Q} \rceil \ge 0$ so that
        \begin{equation}\label{e:nakQ} -Q \le na-kQ \le 0.\end{equation}
    Then rearranging \eqref{eq:bezoutOne} we have 
        \[ nb+kP = \frac{n}{Q} + \frac{P}{Q} (na-kQ).\]
    However, applying \eqref{e:nakQ},  
        \[ -P < nb+kP \le \frac{n}{Q} \leq P.\]
    Thus $na-kQ$ and $nb+kP$ are the desired $a,b$.
\end{proof}
\subsection{Higher dimensions}\label{s:higher}
Let us explain how these results extend to higher dimensions as well. We will prove the main proposition, Proposition \ref{p:main} below, in every dimension $d \ge 2$. Here, we explain how to extend the definition of comb GCC inductively to higher dimensions. To introduce the ideas, given a function $a:\mathbb R^d \to [0,1]$, $\theta \in \mathbb S^{d-1}$ we denote by $\theta a$ the composition of $a$ with a rotation of $\mathbb R^d$ which maps $(0,\ldots,0,1)$ to $\theta$.  Then for $M>0$ we could say $a$ satisfies the $M$ comb GCC in direction $\theta$ if 
    \[ a_{\theta,M}(x) = \inf_{t \in \mathbb R} \frac{1}{M}\int_t^{t+M} \theta a(x,y) \, dy, \quad x \in \mathbb R^{d-1}, \]
satisfies the comb GCC (as a function on $\mathbb R^{d-1}$). As a consequence of Lemma \ref{l:rd-comb} below, comb GCC coincides with relative density in $d=1$, so this definition coincides with the definition for $d=2$, Definition \ref{def:combGCCdird=2}.

The difference in higher dimensions however is that we must impose a quantitative version of comb GCC in $\mathbb R^{d-1}$ since if $d \ge 3$ then $a_{\theta,M}$ itself may possess ``good'' and ``bad'' directions. When $d=2$, $a_{\theta,M}$ only had two (one) directions so no effective covering of the sphere was needed for $a_{\theta,M}$---all coverings were effective. In higher dimensions this will be required, and will be a bit involved. Our goal will be to define a function $\Lambda(a;\rho)$ which, given $a : \mathbb R^d \to [0,1]$, tells us we can perform an effective $(\rho,\lambda)$ covering for $\lambda \ge \Lambda(a;\rho)$. Ultimately, we will say $a$ satisfies the comb GCC if $\Lambda(a;\rho)$ is finite for every $\rho>0$.

First, when $d=0$ we interpret functions $a: \mathbb R^0 \to [0,1]$ as simply a number in $[0,1]$ and set $\Lambda(a;\rho)=\frac{1}{a}$. 

Let us assume we have defined the function $\Lambda(a;\rho)$ for functions $a:\mathbb R^{d-1} \to [0,1]$ and $\rho>0$. Then, we define it for functions $a:\mathbb R^d \to [0,1]$ as follows.
Given $\rho>0$ and $\lambda \ge 1$ we say $a$ has a $(\rho,\lambda)$ effective covering if there exists a set of effective directions $\Theta \subset \mathbb S^{d-1}$, and for each $\theta \in \Theta$ a cap width $\ep_{\theta}$ and length $M_{\theta}$, such that 
\begin{itemize}
    \item $\displaystyle \mathbb S^{d-1} \subset \bigcup_{\theta \in \Theta} \ccap_\theta, \qquad \ccap_\theta = \left\{ \phi \in \mathbb S^{d-1} : \abs{ \phi - \theta } \le \ep_\theta\right\},$
    \item
    \begin{equation}\label{e:comb-eff-2}\lambda_{\Theta}(\rho):=\sup_{\theta \in \Theta} \Lambda(a_{\theta,M_\theta};\rho)<\infty, \quad \text{ with } \ep_\theta M_\theta + (\ep_\theta \lambda)^{-1} < \rho.\end{equation}
\end{itemize}
Now, we define
    \[ \Lambda(a;\rho) = \inf\left\{ \lambda \ge \lambda_\Theta(\rho) : \text{$a$ has a $(\rho,\lambda)$ effective covering $\Theta$} \right\}.\]
We say $a$ satisfies the comb GCC if $\Lambda(a;\rho)$ is finite for every $\rho>0$.

To see this is equivalent to Definition \ref{def:combGCCd=2} when $d=2$, it suffices to show, when $d=1$, that comb GCC is equivalent to relative density. This will follows from Lemma \ref{l:rd-comb} below. But first, let us note some elementary properties of $\Lambda$, that follow directly from the definition and an inductive argument.
\begin{lemma}\label{l:Lambda}
$\Lambda$ is a decreasing function of $a$ and $\rho$. $\Lambda(\theta a;\rho) = \Lambda(a;\rho)$ for every $\theta \in \mathbb S^{d-1}$. If $\ep \in (0,1]$ then $\Lambda(\ep a;\rho)= \frac{1}{\ep} \Lambda(a;\rho)$.
\end{lemma}
Now, we show that GCC implies comb GCC and comb GCC implies relatively density. Since the GCC is equivalent to relative density when $d=1$, this shows comb GCC is equivalent to relative density in that case. 
\begin{lemma}\label{l:rd-comb} For each $d \in \mathbb{N}_0$, and all $L,\eta,\rho>0$ and $a:\mathbb R^d \to [0,1]$,
\begin{itemize} 
    \item If $a$ satisfies the $(L,\eta)$ GCC, then 
    \[ \Lambda(a;\rho) \le \max \left\{ \frac 1{\eta}, \frac{4L}{\rho^2} \right\}.\]
    \item If $\Lambda(a;\rho)<\infty$
    then $a$ is $(L_d,\eta_d)$ relatively dense with 
    \begin{equation}
        L_d = \frac{\sqrt{d!} \rho^2(\Lambda(a;\rho)+1)}{4}, \quad \eta_d = \frac{1}{\Lambda(a;\rho)} \prod_{j=1}^d j^{-\frac{j}{2}}.
    \end{equation}
\end{itemize}

\end{lemma}
\begin{proof} First, if $a$ satisfies the $(L,\eta)$ GCC then for every $\theta \in \mathbb S^{d-1}$, $a_{\theta,L} \ge \eta$. Since $\Lambda(\eta;\rho)=\frac 1{\eta}$, by Lemma \ref{l:Lambda}
    \[ \Lambda(a_{\theta,L};\rho) \le \frac{1}{\eta}.\] 
Then, we may take any covering $\Theta$ with $\ep_\theta= \frac{\rho}{2L}$ and $M_\theta=L$. This covering is $(\rho,\lambda)$ effective as soon as $\lambda \ge \frac{4L}{\rho^2}$.

We will prove the second statement by induction, noting that the case $d=0$ is trivial since $\Lambda(a;\rho)=\frac{1}{a}<\infty$ implies $a>0$ and a constant $a$ is $(L,a)$ relatively dense for all $L>0$, so long as $a>0$.
When $d=1$, if $\Lambda(a;\rho) <\infty$, there exists $\lambda =\Lambda(a;\rho)+1$ such that $a$ has a $(\rho, \lambda)$ effective covering. So for some $\ep, M>0$ with $\ep M + (\ep \lambda)^{-1} <\rho$ we have 
\begin{equation}
    \frac{1}{a_{M}}=\Lambda(a_{M}; \rho) \leq \Lambda(a;\rho).
\end{equation}
That is $a$ is $(M, \frac{1}{\Lambda(a;\rho)})$ relatively dense. Rearranging the expression for $M$, and applying Young's inequality to $\frac{\rho}{\ep} = \frac{\rho \sqrt{\lambda}}{\sqrt{2}} \frac{\sqrt{2}}{\ep \sqrt{\lambda}}$ we have
\begin{equation}\label{eq:thicknessEquivM}
    M < \frac{\rho}{\ep} - \frac{1}{\ep^2 \lambda} \le \frac{\rho^2 \lambda}{4} + \frac{1}{\ep^2\lambda} - \frac{1}{\ep^2\lambda } = \frac{\rho^2\lambda}{4} = \frac{\rho^2(\Lambda(a,\rho)+1)}{4}.
\end{equation}
So $a$ is $(\frac{\rho^2(\Lambda(a;\rho)+1)}{4}, \frac{1}{\Lambda(a;\rho)})$ relatively dense.

Now, the main inductive ingredient we will use is that, by Fubini's theorem, if $a_{\theta,N}$ is $(N,\varrho)$ relatively dense, then $a$ is $(\sqrt{d} N,\varrho d^{-d/2})$ relatively dense. 
Indeed, if $Q$ is a cube of side length $\sqrt{d}N$, then it contains a rotated cube $\theta P$ of side length $N$; see Figure \ref{fig:slicing}. Then,
    \[ \int_Q a \ge \int\limits_{[x_0,x_0+N]^{d-1}} \int\limits_{ [y_0,y_0+N]} \theta a(x,y) \, dy \, dx \ge N\int\limits_{[x_0,x_0+N]^{d-1}} a_{\theta,N}(x) \, dx \ge \varrho N^{d}, \]
and $|Q| = d^{d/2}N^d$ proves the intermediate claim.

\begin{figure}[h]
    \label{fig:slicing}
    \begin{tikzpicture}
        \draw[rotate around={30:(0,0)}] (-2,-2) rectangle (2,2);
        \draw (0,0) circle (2.9);
        \draw[rotate around={-3:(0,0)}] (-2.9,-2.9) rectangle (2.9,2.9);
        \draw node at (-2.5,-2.3) {$Q$};
        \draw node at (-1.5,-0.7) {$N$};
        \draw node at (-1.2,1.4) {$N$};
        \draw node at (0,0) {$\theta P$};
        \draw node at (3.7,0) {$\sqrt{d} N$};
    \end{tikzpicture}
    \caption{The square $\theta P$ is contained within the square $Q$.}

\end{figure}
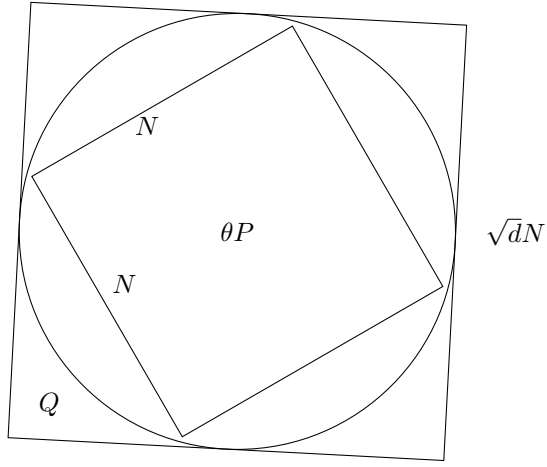

So consider $a:\Rb^d \ra [0,1]$ with $\Lambda(a; \rho)<\infty$ and assume the result for functions on $\Rb^{d-1}$. Then for $\lambda=\Lambda(a;\rho)+1$, $a$ has a $(\rho,\lambda)$ effective covering. So for some $\theta \in \mathbb{S}^{d-1}$ and $M$
\begin{equation}
    \Lambda(a_{\theta,M};\rho) \leq \Lambda(a,\rho).
\end{equation}
Then by the inductive hypothesis $a_{\theta,M}$ is $(L_{d-1},\eta_{d-1})$ relatively dense, and by \eqref{eq:thicknessEquivM}, $M \leq \frac{\rho^2(\Lambda(a,\rho)+1)}{4} \leq L_{d-1}$ so $a_{\theta, L_{d-1}}$ is relatively dense as well. Then by our Fubini argument $a$ is $(\sqrt{d}L_{d-1}, \eta_{d-1} d^{-d/2})=(L_d,\eta_d)$ relatively dense. 
\end{proof}

\section{Proof of comb GCC uncertainty principle}\label{s:comb-proof}
\subsection{Preparatory Lemmas}\label{s:prep}
The general idea of the propagation of singularities approach it to utilize a transfer function 
	\[ A_\theta(x',x_d) = \int_0^{x_d} a_{\theta,M}(x') - \theta a(x',t) \, dt, \quad x' \in \mathbb R^{d-1}, \ x_d \in \mathbb R \]
to switch between $a$ and the comb GCC function $a_{\theta,M}$ which is invariant in the direction $\theta$. This propagation approach is inspired by normal form methods, as in \cite{BZ12, Sun23} and references therein. The main technical obstacle appearing is that this function may not be bounded, even if $a$ satisfies the comb GCC in direction $\theta$. Indeed, consider a smooth function $b:\mathbb R \to [0,1]$ which has very small average, say $0.1$, on intervals $[2^k,2^k+1]$ and yet equals $1$ everywhere else. For such a function, its transfer function
	\begin{equation}\label{e:bt} \int_0^y \left[ b(t)-b_0 \right]\, dt \end{equation}
is unbounded for any $b_0 \in \mathbb R$. Indeed, the best choice is $b_0=1$ but then we have a lower bound $\gtrsim \log|y|$. Nonetheless, we can construct a minorant of $b$ which remains relatively dense and for which the transfer function \eqref{e:bt} is bounded. Simply take $b_1=0.1$ outside $[2^k,2^k+1]$ and $b_1=b$ on $[2^k,2^k+1]$. Of course such a $b_1$ will not be smooth, but in fact this can also be accomplished with a bit more care. This is the main idea of the next Lemma.

To clearly control all the constants, we introduce one more piece of notation. A function $\omega:[0,\infty) \to [0,\infty)$ which is monotone increasing, continuous at $0$ and satisfies $\omega(0)=0$ is called a modulus of continuity. A function $a$ is said to be $\omega$-continuous if
    \[ |a(x)-a(y)| \le \omega(|x-y|).\] 
Clearly an $\omega$-continuous function is uniformly continuous, and given a uniformly continuous function $a$, we may define
    \[ \omega_a(t) = \sup_{x,y :|x-y| \le t} |a(x)-a(y)|\]
which is a modulus of continuity and $a$ is $\omega_a$-continuous. Furthermore, note that $a$, $\theta a$ and all $a_{\theta,M}$ are all $\omega_a$-continuous.
\begin{proof}
    Define $a_{M}(x;t) = \frac{1}{M}\int_t^{t+M} a(x,x_d) \, d x_d $. Clearly $a_{M}(\cdot;t)$ is also $\omega_a$ continuous for each $t \in \mathbb R$. Now fix $x,y \in \mathbb R^{d-1}$. We may assume $a_M(x) \le a_M(y)$. Letting $\ep>0$, we can find $t
    _0$ such that
         \[ |a_M(x;t) - a_M(x)| < \ep, \quad a_M(x) = \inf_{t \in \mathbb R} a_M(x;t).\]
    But then,
        \[ a_M(y) \le a_M(y;t_0) \le a_M(x;t_0) + \omega_a(|x-y|)
        \le a_M(x) + \ep + \omega_a(|x-y|).\]
    Since $\ep>0$ was arbitrary and $a_M(x) \le a_M(y)$, we have \[|a_M(x)-a_M(y)| \le \omega(|x-y|).\]
\end{proof}

We now state the lemma which allows us to replace a uniformly continuous function satisfying the comb GCC by a smooth minorant that still satisfies the comb GCC, and for which the associate transfer function ($A$ in (iii) below) is bounded. We only consider the direction $\theta=(0,\ldots,0,1)$ for simplicity, but of course applying the result to $\theta a$ will handle the general case. For simplicity, we denote
    \[ a_M(x') = \inf_{t \in \mathbb R} \frac{1}{M}\int_t^{t+M} a(x',x_d) \, d x_d.\]

\begin{lemma}\label{l:extract}
Let $\omega$ be a modulus of continuity and $d \in\mathbb{N}_0$. There exists $C,\delta,c>0$ such that for all
$a':\mathbb R^d \to [0,1]$ which are $\omega$-continuous and for which $\Lambda((a')_M;\rho) \le \Lambda(a';\rho)$ there exists $a:\mathbb R^d \to [0,1]$, $\overline{a}:\Rb^{d-1} \ra [0,1]$, and $X \subset \mathbb R^{d}$ such that 
\begin{itemize} 
\item[(i)] $a\le \mathbbm{1}_X$, $\delta \mathbbm{1}_{X_\delta} \le a'$, so in particular $a \leq \d^{-1} a'$
        \item[(ii)] $\Lambda(\overline{a};\rho)  \le C \Lambda(a';c\rho)$, 
        \item[(iii)] 
        the transfer function
        \begin{equation}
        A(x',x_d) = \int_0^{x_d} \overline{a}(x') - a(x',t)\, dt,
        \end{equation}
        as well as $a$ and $\overline{a}$, satisfy
        for each $m \in \mathbb{N}_0$,
            \begin{equation}\label{e:a-smooth-bounded} |\nabla^m f| \le C^{m+1}, \quad f \in \{A/M,a,\overline{a}\}.\end{equation}
        In particular, $a$ and $\overline{a}$ are $\omega$-continuous with $\omega(t) = C^2t$.
    \end{itemize}
    
\end{lemma}
The proof is postponed until Section \ref{s:structure} below.

\subsection{Propagation of singularities}
Now, we use the structural lemma (Lemma \ref{l:extract}) to prove the following main resolvent estimate. 

\begin{proposition}\label{p:main}
Let $\omega$ be a modulus of continuity, $d \in \mathbb{N}_0$, and $\Lambda_0 \ge 1$. There exists $C_1 \ge 1$ and $c_1\in (0,1)$ such that 
\begin{itemize}
    \item for all $a':\mathbb R^d \to [0,1]$ which are $\omega$-continuous and $\Lambda(a;1) \le \Lambda_0$.
    \item all $\lambda \ge C_1\Lambda(a';c_1)$
    \item and all $v \in H^2(\mathbb R^D)$ for $D \ge d$
\end{itemize} 
there holds
    \begin{equation}\label{e:p-ind} c_1 \norm{v}_{L^2(\mathbb R^D)}^2 \le \norm{(-\Delta-\lambda^2) v}_{L^2(\mathbb R^D)}^2 + \ip{av}{v}_{L^2(\mathbb R^D)}
    .\end{equation}
\end{proposition}
Note that Theorem \ref{thm:comb} in any dimension follows immediately from this Proposition, Lemma \ref{l:lowFreqResolve} and Theorem \ref{thm:Miller}.

We prove the Proposition by induction using a positive commutator argument involving the modified $a$ constructed in Lemma \ref{l:extract}. Let us call P\ref{p:main}$(d)$ Proposition \ref{p:main} in the case of dimension $d$. Notice that the base case P\ref{p:main}$(0)$ is trivial since in that case $a$ is a positive constant and $\eta=a$. Furthermore, $C_1=1$, $c_1 = \eta = \Lambda_0^{-1}$. Supposing P\ref{p:main}$(d-1)$ holds, we will use the following proposition to lift to P\ref{p:main}$(d)$. We only state and prove the case of $\theta=(0,\ldots,0,1) \in \mathbb S^{d-1}$ but the general case can be reduced to this one simply by considering the rotation $\theta a$ rather than $a$.

\begin{proposition}\label{p:main-microlocal}
Let $\omega$ be a modulus of continuity, $d \in \mathbb{N}_0$, and $\Lambda_0' \ge 1$. There exists $C_2 \ge 1$ and $c_2,\delta \in (0,1)$ such that 
\begin{itemize}
    \item for all $a':\mathbb R^d \to [0,1]$ which are $\omega$-continuous and $\Lambda(a';1) \le \Lambda_0'$, 
    \item all $\lambda \ge C_2\Lambda(a';c_2)$ such that $\theta_0=(0,\ldots,0,1)$ is a $(c_2,\lambda)$ effective direction for $a'$ with length $M_0$ and cap width $\ep_0$,
    
    \item and all $v$ with 
	\[ \supp \hat v \subset \{ (\xi',\xi_d) \in \mathbb R^{d-1} \times \mathbb R : \ |\xi'| \le 3 \ep_0\lambda, \lambda (1-3\ep_0^2) \le |\xi_d| \le \lambda(1+3\ep_0^2) \},\]
    
\end{itemize} 
we have
    \begin{equation}\label{e:p-sector} c_2 \ltwo{v}^2 \le 
    \ltwo{(-\Delta-\lambda^2) v}^2 +
     \norm{v}_{L^2(X)}^2, \quad \delta \mathbbm{1}_{X_\delta} \le a' .\end{equation}
\end{proposition}

We will prove the following two lemmas which by induction will establish Proposition \ref{p:main}. 
\begin{lemma}\label{l:ind} P\ref{p:main}$(d-1)$ implies P\ref{p:main-microlocal}$(d)$. \end{lemma}
\begin{lemma}\label{l:almost-2} P\ref{p:main-microlocal}$(d)$ implies P\ref{p:main}$(d)$.
\end{lemma}
We begin with the latter, as it follows the approach in the proof of Theorem \ref{thm:comb} above when $d=2$, although we require an additional step when $a$ is invariant in at least one direction. 

\begin{proof}[Proof of Proposition \ref{p:main}(d) assuming Proposition \ref{p:main-microlocal}(d) (Lemma \ref{l:almost-2})]
First we work in the case $D=d$. We set $C_1=C_2$ and will select $c_1$ depending on $\delta$ and $c_2$ from P\ref{p:main-microlocal}$(d)$ and $C_{d,\delta}$ from Lemma \ref{l:almost}. Preliminarily, suppose
    \begin{equation}\label{e:c1FirstChoice} 0<c_1\le  \sqrt{\frac{c_2}{4C_{\d,d} (c_2+2)} }. \end{equation} 
We also assume that $\hat v$ is supported in 
    \[ A_\lambda = \left\{ \xi \in \mathbb R^d : \lambda-\varkappa \lambda^{-1} \le |\xi| \le \lambda + \varkappa \lambda^{-1} \right\},\] 
for some $\varkappa>0$ small, to be determined. Now, for each $\lambda \ge \Lambda(a';c_1)$ we can find an $(c_1,\lambda)$ effective covering of the sphere $\Theta$, which tells us the annulus $A_\lambda$ can be covered by wedges $\text{Wedge}_\theta$, of width $\ep_\theta$. Let us denote by $\text{Wedge}_\theta^*$ the same wedge but of three times the width. 
If $\hat v$ is supported in $A_\lambda \cap \text{Wedge}_\theta^*$, then for $\varkappa$ smaller than some dimensional constant, $\supp \widehat{\theta v}$ satisfies the condition of Proposition \ref{p:main-microlocal} and $\theta_0$ is an effective direction for $\theta a'$. Note $\theta a'$ is $\omega$ continuous and $\Lambda(\theta a';1)=\Lambda(a';1) \leq \Lambda_0$. Thus, by Proposition \ref{p:main-microlocal}$(d)$, for all $\lambda \ge C_2 \Lambda(a;c_2)$, we obtain
    \[ c_2 \norm{v}_{L^2}^2 \le \ltwo{(-\Delta-\lambda^2)v}^2+\ip{\mathbbm{1}_Xv}{v}, \quad \supp \hat v \subset A_\lambda \cap \text{Wedge}_\theta^*, \quad \d \mathbbm{1}_{X_{\d}} \leq a'.\]
Note we have $\ltwo{(-\Delta-\lambda^2)v}^2 \leq C \varkappa^2 \ltwo{v}^2$, so potentially taking $\varkappa$ smaller we may absorb that term back into the left hand side and obtain
\[ \frac{c_2}{2} \norm{v}_{L^2}^2 \le \ip{\mathbbm{1}_Xv}{v}.\]

Let again $\Xi$ be a collection of $1/2$ separated points on $\lambda \mathbb S^{d-1}$ with gaps no greater than 1. Notice that for any $\xi \in \Xi$, there exists $\theta \in\Theta$ such that $\xi \in \text{Wedge}_\theta$ and hence
    \[ d(\xi, A_\lambda \backslash \text{Wedge}_\theta^*) \ge (\lambda-\varkappa \lambda^{-1}) \sin(\ep_\theta) \geq  \frac{\ep_\theta \lambda}{4} > \frac{1}{4c_1}.\]
Therefore, we can apply Lemma \ref{l:almost}, and use our initial restriction on $c_1$ from \eqref{e:c1FirstChoice}, to conclude that for $\hat v$ supported in $A_\lambda$
   \[  \frac{c_2}{c_2+2} \ltwo{v}^2 \le \frac{2C_{\d,d}}{\d} \ip{a'v}{v} + \frac{c_2}{2(c_2+2)} \ltwo{v}^2. \]
Now, appealing to Lemma \ref{l:equiv}, for any $v \in H^2$ we obtain
    \begin{equation}\label{e:final-d}\frac{c_2 \delta}{4 C_{d,\d}(c_2+2)} \norm{v}_{L^2(\Rb^d)}^2 \le C \norm{(-\Delta-\lambda^2) v}_{L^2(\Rb^d)}^2 + \ip{a'v}{v}_{L^2(\Rb^d)}.\end{equation}

Thus we arrive at an expression for $c_1$ in \eqref{e:p-ind}, as long as we take into account \eqref{e:c1FirstChoice} as well.

Finally, to extend to $D > d$, we use a standard partial Fourier transform argument. Let $\mathcal F_{2}$ be the Fourier transform in the final $D-d$ variables, i.e.
	\[ \mathcal F_{2} f(x,\zeta) = \int_{\mathbb R^{D-d}}  e^{-i y \cdot \zeta} f(x,y) \, d y, \quad x \in \mathbb R^d \ ,\zeta \in \mathbb R^{D-d}.\] 
For each $\zeta \in \mathbb R^{D-d}$, define the functions, for $x \in \mathbb R^d$, $v_\zeta(x) = \mathcal F_{2} u(x,\zeta)$. By standard identities for the Fourier transform, on $\Rb^d$ we have 
	\begin{equation}\label{e:IBP}(-\Delta_{\mathbb R^d} -\lambda^2 + |\zeta|^2) v_\zeta(\cdot) = \mathcal F_2 \left[ (-\Delta-\lambda^2 )u \right]( \cdot,\zeta).\end{equation}
On the other hand, \eqref{e:final-d} states
	\[ c_1 \norm{ v_{\zeta} }_{L^2(\mathbb R^d)}^2 \le \norm{ (-\Delta_{\mathbb R^d} -\lambda^2 + |\zeta|^2) v_\zeta }_{L^2(\mathbb R^d)}^2 + \ip{a'v_\zeta}{v_\zeta} . \]
Using \eqref{e:IBP}, integrating over $\zeta \in \mathbb R^{D-d}$, and using Plancherel's theorem establishes
\[ c_1 \norm{ u }_{L^2(\mathbb R^D)}^2 \le \norm{ (-\Delta-\lambda^2)  u }_{L^2(\mathbb R^D)}^2 + \ip{a'u}{u}_{L^2(\Rb^D)} . \]

\end{proof}

\begin{proof}[Proof of Proposition \ref{p:main-microlocal}  assuming Proposition \ref{p:main} for $d-1$ (Lemma \ref{l:ind})]
Let $C,\delta,c>0$ be the constants from Lemma \ref{l:extract}. Then set $\omega'(t)=C^2t$ and $C_1,c_1$ be the constants provided by P\ref{p:main}$(d-1)$ with $\omega=\omega'$ and $\Lambda_0=C \Lambda_0'$. Then, we set
    \begin{equation}\label{e:c2Choice} c_2  = \frac{cc_1}{2(4+3C^2+C^3)}, \quad C_2 = C C_1.\end{equation}
Now, since $\theta_0$ is an $(c_2,\lambda)$ effective direction, with length $M$ and cap width $\ep_0$ then we have 
    \[ \Lambda((a')_{M};c_2) \le \Lambda(a';c_2), \quad \ep_0M + (\ep_0\lambda)^{-1} < c_2 \] 
for all $\lambda \ge \Lambda(a';c_2)$. Let now $a$, $\overline{a}$, and $X$ be those provided by Lemma \ref{l:extract} and note $\Lambda(\overline{a};c_1) \leq C \Lambda(a';c_2)$ so $a$ satisfies the hypotheses of P\ref{p:main}$(d-1)$ with modulus of continuity $\omega'$ and $\Lambda_0=C \Lambda(a';1)$. Recalling $A(x',x_d)$ from Lemma \ref{l:extract}, let 
    \begin{equation}
        b(x',x_d)=\frac{1}{M} \left( \int_0^{x_d} \overline{a}(x') - a(x',t) dt\right) = \frac{1}{M} A(x',x_d),
    \end{equation}
    and denote by $B$ the operator on $L^2$ of multiplication by $b$. By Lemma \ref{l:extract} $\ltwo{bu} \le C\ltwo{u}$. We will compute the commutator $[B,(-\Delta-\lambda^2)]$ in two different ways. First, using the self-adjointness of $B$ and $\Delta$, 
    \begin{align}\label{eq:commutatorSelfAdjoint}
       \abs{\<[B,(-\Delta-\lambda^2)] v,v\>} &= \abs{ \<Bv, (-\Delta-\lambda^2) v\> - \<(-\Delta-\lambda^2) v,Bv\>}\\
       &\le 2C \ltwo{(-\Delta-\lambda^2) v} \ltwo{v}.
    \end{align}
On the other hand, using the product rule, 
    \begin{equation}\label{eq:commutatorProductRule}
        [B,(-\Delta-\lambda^2)]v = 2\nabla b \cdot \nabla v + (\Delta b) v.
    \end{equation} 
Combining together \eqref{eq:commutatorSelfAdjoint} and \eqref{eq:commutatorProductRule}, rearranging and taking absolute values, then applying Young's inequality, we have 
\begin{align}
    |\<\p_{x_d} b \p_{x_d} v,v \>| &\leq C\ltwo{(-\Delta-\lambda^2) v}\ltwo{v}  + |\<\nabla_{x'} b \nabla_{x'} v,v \>| + \frac{1}{2}|\<(\Delta b) v, v\>|\\
    &\leq C^2\ltwo{(-\Delta-\lambda^2)v}^2 + \ltwo{\nabla_{x'}b \nabla_{x'}v}^2 + \ltwo{(\Delta b)v}^2 + \ltwo{v}^2. \label{eq:commutatorIntermed}
\end{align}
We seek to control the terms on the right hand side as errors, and to bound the term on the left hand side from below.  First, the term $\nabla_{x'} b \cdot \nabla_{x'} v$ will be treated as an error of the order $\ep_0 \lambda$ since $\nabla_{x'} b$ is bounded and by Plancherel's theorem and the frequency support of $v$,
	\begin{equation}\label{eq:xprimeDeriv}
	    \norm{\nabla_{x'} v}_{L^2} = \ltwo{\xi' \hat{v}} \le 3\ep_0\lambda \ltwo{\hat{v}} = 3\ep_0 \lambda \norm{v}_{L^2}.
	\end{equation}
For the derivative in $x_d$, we use the fundamental theorem of calculus to write
	\[ \partial_{x_d} b(x',x_d) = \frac{\overline{a}(x') - a(x',x_d)}{M}.\]
Furthermore, using the support of $\hat v$ a second time we see $\ltwo{\partial_{x_d} v -\lambda v} \le 3 \ep_0^2 \lambda \ltwo{v}$. Therefore we have
	\begin{align} \frac{\lambda}{M} \ip{\overline{a} v}{v}  &\le |\<\partial_{x_d} b \partial_{x_d} v, v\>| + \frac{\lambda}{M} \ip{a v}{v} + 3 \ep_0^2 \lambda \ltwo{v}^2.
    \end{align}
    Now combining this with \eqref{eq:commutatorIntermed} and \eqref{eq:xprimeDeriv} we obtain 
    \begin{align}
		\frac{\lambda}{M} \ip{\overline{a} v}{v} \le &C^2\ltwo{(-\Delta-\lambda^2) v}^2 + \frac{\lambda}{M} \ip{a v}{v}  \\
        &+ (1+3 \ep_0^2 \lambda  + \norm{\Delta b}_\infty + \norm{\nabla_{x'} b}_\infty 3\ep_0 \lambda )\ltwo{v}^2. \end{align}
Now note by Lemma \ref{l:extract}, $\norm{\Delta b}_\infty \le C^3$ and $\norm{\nabla_{x} b}_\infty \le C^2$. So, dividing through by $\frac{\lambda}{M}$ in the above display we obtain
\begin{align}
\label{e:step1} \ip{\overline{a} v}{v} \le &\frac{C^2M}{\lambda} \norm{(-\Delta-\lambda^2) v}^2 +  \ip{a v}{v} \\
    &+ \left[ \frac{M(1 + C^3)}{\lambda} + 3M\ep_0^2+ 3C^2M\ep_0 \right] \ltwo{v}^2.
\end{align}
To control the error, we invoke the assumption that $\theta_0$ is a $(c_2,\lambda)$ effective direction and our choice of $c_2$ in \eqref{e:c2Choice} to ensure
    \begin{equation} \label{eq:effectiveDirectionUse}
    \frac{M(1 + C^3)}{\lambda} + C^2M\ep_0 + 3M\ep_0^2 \le c_2(4+3C^2+C^3) = \frac{c_1}{2}.
    \end{equation}
Now, for all $\lambda \ge CC_1\Lambda(a';c_2) \ge C_1\Lambda(\overline{a};c_1)$, 
by the induction hypothesis applied to $\overline{a}$
, we have
\begin{equation}
    c_1 \norm{v}_{L^2(\Rb^d)}^2 \leq \ltwo{(-\Delta-\lambda^2) v}^2 + \<\overline{a} v,v\>_{L^2(\Rb^d)}.
\end{equation}
Combining this with \eqref{e:step1} and \eqref{eq:effectiveDirectionUse} we have 
    \begin{align}
    c_1 \ltwo{v}^2 
    &\le 2\ltwo{(-\Delta-\lambda^2) v}^2 + \ip{av}{v} + \frac{c_1}{2} \ltwo{v}^2.
    \end{align}
Finally, absorbing the error and noting that \eqref{e:c2Choice} implies $\frac{c_1}{4} > c_2$, we obtain \eqref{e:p-sector}.
\end{proof}

\subsection{Structure of comb GCC sets}\label{s:structure}
We will prove our main structural claim by induction. The main idea is that if $a$ satisfies the comb GCC in a direction $\theta_0=(0,\ldots,0,1)$, then in small strips $B(x_0',\delta) \times \mathbb R$, $a(x',x_d)$ can be bounded below by a product $a_1(x')a_2(x_d)$ with $a_2$ having a bounded transfer function and being relatively dense, while $a_1$ satisfies the comb GCC on $\mathbb R^{d-1}$.

\begin{definition}\label{d:b}
    We say a function $b:\mathbb R \to \mathbb R$ has almost periodic density if there exists $\varrho \in \mathbb R$ and an increasing sequence of real numbers $\{s_k\}_{k \in \mathbb Z}$ satisfying 
    \[ s_{k+1}-s_k \le M, \quad \lim_{k \to \pm \infty} s_k = \pm \infty,\]
    such that
        \[ \frac{1}{s_{k+1}-s_k} \int_{s_k}^{s_{k+1}} b(t) \, dt = \varrho, \quad k \in \mathbb Z.\]
    When we specify the parameters, we say $b$ has almost $M$-periodic density $\varrho$.
\end{definition}

\begin{lemma}\label{l:B}
If $b:\mathbb R \to [0,\infty)$ has almost $M$-periodic density $\varrho$, then
    \[ B(y) = \int_0^y \left[b(t) - \varrho\right] \, dt \]
is a bounded function with $|B| \le 4M \norm{b}_\infty$. Moreover, if $\varrho >0$ and $\inf_{k} s_{k+1}-s_k \ge m$, then $b$ is $(2M,\frac{\varrho m}{2M})$ relatively dense.
\end{lemma}
\begin{proof}
Notice that $|\varrho| \le \norm{b}_\infty$ follows immediately from the definition of $\varrho$.
Now, let $s^*= \min\{ s_k : 0 \le s_k\}$. The key property of $B$ is that $B(s_k)=B(s^*)$ for each $k \in \mathbb Z$. 
Furthermore, for any $y \in \mathbb R$ there exists a unique $k \in \mathbb Z$ such that $s_k < y \le s_{k+1}$. Therefore,
    \[ \abs{B(y)-B(s^*)} = \abs{B(y) - B(s_k)} = \abs{\int_{s_k}^{y} \left[ b(t) - \varrho\right] \, dt} 
    \le 2M \norm{b}_\infty.\] 
Finally, $|B(s^*)| \le 2M \norm{b}_\infty$ so the first claim follows from the triangle inequality. To prove the second claim, notice that if $s_{k-1} < t \le s_{k}$ then
    \[ \int_t^{t+2M} b(y) \, dy \ge \int_{s_k}^{s_{k+1}} b(y) \, dy = (s_{k+1}-s_k) \varrho \ge m\varrho.\]
\end{proof}

\begin{lemma}\label{l:smoothing}
There exists a constant $C>0$ such that the following holds for any $M,\varrho,\delta>0$ with $\delta < M/2$. If
    \[ Y = \bigcup_{j \in \mathbb Z} (y_j-\delta,y_j+\delta) \subset \mathbb R\]
is a disjoint union of open balls which is $(M,\varrho)$ relatively dense, then, we can find $a \le \mathbbm 1_{ Y}$ and $\eta \ge \frac{\varrho}{4}$ such that 
\begin{itemize}
    \item[(i)] $a$ is $(2M,\frac{\varrho}{8})$ relatively dense.
    \item [(ii)] the transfer function
        \begin{equation}
        A(y) = \int_0^y \left[a(t)- \eta \right]\, dt,
        \end{equation}
        as well as $a$ both satisfy for each $m\in \mathbb{N}_0$ we have 
            \begin{equation} |\partial^m f| \le C^{m+1} ( \rho\delta)^{-m}, \quad f \in \{A/M,a\},\end{equation}
\end{itemize}
\end{lemma}

\begin{proof} 
Let $s_0=\inf\{ 0 \le y : y \not \in Y \}$. Inductively, define for $k \in \mathbb N$,
    \[s_k = \inf\{ y \ge s_{k-1} + M: y \not \in Y \} , \quad s_{-k} = \sup \{ y \le s_{-(k-1)} - M: y \not \in Y\}.\] 
Clearly $s_{\pm k} \to \pm \infty$ and $M\leq s_{k+1}-s_k \leq M + 2\delta \le 2M$ due to the fact that $Y$ consists of disjoint intervals of length $2\delta \le M$. 

Since $Y$ is $(M, \varrho)$ relatively dense, we know that
    \[ 1 \ge \frac{|Y \cap (s_k,s_{k+1})|}{|(s_k,s_{k+1})|} \ge \frac{|Y\cap (s_k,s_k+M) |}{2M} \ge \frac{\varrho }{2}.\] 
Listing out all the elements we have 
    \[ Y \cap (s_k,s_{k+1}) = \bigcup_{i=1}^{J_k} (y_{k,i}-\delta,y_{k,i}+\delta).\]
Now, for each $k$ set
    \begin{equation}\label{e:t}t_{k} = \frac{\varrho}{2} \cdot \frac{|(s_k,s_{k+1})|}{|Y\cap (s_k,s_{k+1}) |} \in \left[ \frac{\varrho}{2},1\right],\end{equation}
then we define
    \[ \tilde Y = \bigcup_{k\in \mathbb{Z}} \bigcup_{i=1}^{J} (y_{k,i}-t_{k}\delta,y_{k,i}+t_{k}\delta), \]
which satisfies $|\tilde Y \cap (s_k,s_{k+1})| = t_{k} |Y \cap (s_k,s_{k+1})|= \tfrac{\varrho}{2} |(s_k,s_{k+1})|$. 

Now note that for any interval $I \subset \Rb$ we can find a function $b_I$ in $C^\infty_0(\Rb)$ such that the following properties hold uniformly over $I$:
\begin{itemize}
    \item[(a)] $\mathbf 1_{\tfrac 12 I} \le b_I \le \mathbf 1_I$. Here $\tfrac 12 I$ is the interval concentric with $I$ of half its length.
    \item[(b)] There exists $c_1>\frac 12$ such that $\int b_I = c_1 \, |I|$.
    \item[(c)] There exists $c_2>0$ such that for each $m  \in \mathbb{N}_0$  \[ \abs{\partial^m b_I} \le c_2^m |I|^{-m}.\]
\end{itemize}
Then we define 
\begin{equation}
    a(y) = \sum_{k \in \mathbb Z} \sum_{i=1}^{J_k} b_{I_{k,i}}(y), \quad I_{k,i} = (y_{k,i}-t_{k}\delta,y_{k,i}+t_{k}\delta).
\end{equation}
By property (b) of $b_I$ and the construction of $\tilde Y$, it is clear that 
    \begin{equation}\label{e:bksk} \frac{1}{s_{k+1}-s_k}\int_{s_k}^{s_{k+1}} a(y) \, dy =  \frac{c_1 \, \varrho}{2}=: \eta \ge \frac{\varrho}{4} .\end{equation}
Hence, by the second statement in Lemma \ref{l:B} $a$ is indeed $(2M, \frac{\varrho}{8})$ relatively dense. Furthermore by property (c) and the lower bound on $t_{k}$ given in \eqref{e:t},
    \[ |\partial_y^m a(y)| \le c_2^m \left( \frac{\varrho}{2} \delta \right)^{-m}\]
Also, this guarantees $\partial_y^{m+1} A$ obeys the same bound--it remains to control $A$ itself. But this is a direct consequence of the first statement in Lemma \ref{l:B}.

\end{proof}

We prove a slightly stronger version of Lemma \ref{l:extract} which is more technical, but easier to induct upon. Lemma \ref{l:extract} is recovered by taking $\delta'=\delta$, $C = \max\{ C,C_{\delta'} \}$, $a=a^1$, $\overline{a}=\overline{a^1}$, and $X=X^1$.

\begin{lemma}\label{l:extractInduct}
Let $\omega$ be a modulus of continuity and $d\in \mathbb{N}_0$. There exists $C,\delta,c>0$ such that 
\begin{itemize}
    \item for all
$a':\mathbb R^d \to [0,1]$ which are $\omega$-continuous,
    \item all $M$ such that $\Lambda((a')_M;\rho) \le \Lambda(a';\rho)$,
    \item and all $0<\delta' \le \delta$,
\end{itemize}
there exist $C_{\delta'} >0$, $a^1,a^2:\mathbb R^d \to [0,1]$, $\overline{a^1}:\mathbb R^{d-1} \to [0,1]$, and $X^1,X^2 \subset \mathbb R^{d}$ such that 
\begin{itemize} 
\item[(i)] $a^i\le \mathbbm{1}_{X^i}$, $ \delta \mathbbm{1}_{X^i_{\delta'}} \le a'$, and $X^i$ is a union of balls of radius $\delta'$,
        \item[(ii)] $\Lambda(\overline{a^1};\rho)\le C \Lambda(a'_M;c\rho)$, $\Lambda(a^2;\rho) \le C\Lambda(a';c\rho)$,
        \item[(iii)] 
        the transfer function
        \begin{equation}
        A(x',x_d) = \int_0^{x_d} \overline{a^1}(x') - a^1(x',t)\, dt,
        \end{equation}
        as well as $a^i$ and $\overline{a^1}$ satisfy: for each $m\in \mathbb{N}_0$
            \begin{equation}\label{e:a-smooth-bounded-2} |\nabla^m f| \le C_{\delta'}^{m+1} , \quad f \in \{A/M,a^i,\overline{a^1}\}.\end{equation}
    \end{itemize}
\end{lemma}

\begin{proof}
The case $d=0$ is trivial since there is nothing to prove. Let $d \ge 1$. We will use a superscript $*$ to denote objects provided by the induction step applied to $a'_{M}$. First, let $\delta^*>0$ and let $\delta'>0$ be any small parameter satisfying
    \begin{equation}\label{e:choiceDeltaPrime} \delta' \le \frac{\delta^*}{16}, \quad \omega(2\delta') \le \frac{\delta^*}{4}.\end{equation}
In this way, we have $X^* = \cup_j B_j \subset \mathbb R^{d-1}$ a collection of balls of radius $\delta'$ and $a^*$ in $C^\infty$ such that
    \begin{equation}\label{e:deltaStar}a^* \le \mathbbm 1_{X^*}, \quad \delta^* \mathbbm 1_{X^*_{\delta^*}} \le a'_{M}, \end{equation}
and $\Lambda (a^*;\rho) \le C^*\Lambda( a'_{M};c^*\rho)$ ($X^*$ and $a^*$ are $X^2$ and $a^2$ from the induction hypothesis). For each $j$, let $x_j'$ be the center of $B_j$ and define 
    \[ Y^j = \left\{ x_d \in \mathbb R : a'(x_j',x_d) \ge \tfrac{\delta^*}{2}\right\}.\] 
First, since $a'$ is $\omega$-continuous, by the second condition in \eqref{e:choiceDeltaPrime}, 
if $x \in B(x_j',2\delta') \times Y^j_{2\delta'}$ then $a'(x) \ge \frac{\delta^*}{4}$. Thus if we take
    \[ X^1 = \bigcup_j B_j \times Y^j_{\delta'},\]
we clearly have the property $\frac{\delta^*}{4} \mathbbm 1_{X^1_{\delta'}} \le a'$ from (i). 

Next, we will bound $\mathbbm 1_{X^1}$ from below by a smooth function $a^1$ satisfying the remaining properties (ii) and (iii).

Towards this end, by same logic as \eqref{e:switch}, $[Y^j]_M \ge \frac{\delta^*}{2}$ and hence the same holds for the larger quantity $[Y^j_{\delta'}]_M$. Now, we may use a standard covering algorithm to replace $Y^j_{\delta'}$ by a disjoint union of balls of radius $\delta'$. Call this set $\overline{Y^j} \subset Y^j_{\delta'}$ and note that it satisfies
    \[ [\overline{Y^j}]_{2M} \ge \frac{\delta^*}{4}.\]
For each $j$, $\overline{Y^j}$ satisfies the conditions of Lemma \ref{l:smoothing} so we now obtain $a_j$, $\eta_j$ and form
    \begin{align} &a^1(x',x_d) = a^*(x') \sum_{j} \mathbbm 1_{B_j}(x') a_j(x_d), \\
    &\overline{a^1}(x') = a^*(x') \sum_j \mathbbm 1_{B_j} \eta_j
    .\end{align}
Since $a^*$ is supported on $X^*$ and $a_j$ on $\overline{Y^j} \subset Y^j_{\delta'}$, it is clear that $a^1$ is supported on $X^1$. It is routine to check that $a^1$,$\overline{a^1}$, and $A$ satisfy the regularity conditions in (iii) using the estimates for $a^*$ provided by the induction hypothesis and the same for $a_j$ provided by Lemma \ref{l:smoothing}. 

Next we to check (ii) for $a^1$. First note by Lemma \ref{l:smoothing}, $a_j$ is $(4M, \frac{\d^*}{64})$ relatively dense, therefore 
    \begin{equation}\label{e:deltaPrime}a_{4M} \ge a^* \frac{\delta^*}{64}, \quad \overline{a^1} \ge a^* \frac{\delta^*}{16}\end{equation}
so that, relying on the decreasing nature of $\Lambda$ from Lemma \ref{l:Lambda} and the induction hypothesis, 
    \begin{align}\label{e:a4MLambda}  &\frac{\delta^*}{64} \Lambda(a_{4M};\rho) 
    \le \Lambda(a^*;\rho)  
    \le C^*\Lambda( a'_{M};c^*\rho) \\
     &\frac{\delta^*}{16} \Lambda(\overline{a^1};\rho) 
    \le \Lambda(a^*;\rho) 
    \le C^*\Lambda( a'_{M};c^*\rho) 
     \end{align}
Thus taking $C \ge \frac{64C^*}{\delta^*}$ and $\delta=\frac{\delta^*}{4}$, we have finished with $a^1$. 

Now, we must construct $X^2$ and $a^2$, but we have already done most of the work. Let now $\lambda \ge \Lambda(a';\rho/4)$ and $\Theta$ be a $(\rho/4,\lambda)$ effective covering for $a'$. Then, for each $\theta \in \Theta$, there exists $M_\theta,\ep_\theta>0$ such that 
    \[ \ep_\theta M_\theta + (\ep_\theta \lambda)^{-1} < \frac{ \rho}{4}.\] 
We look at $a'_{\theta,M_\theta}= (\theta a')_{M_\theta}$. Now by what we have proved so far, concluding with \eqref{e:a4MLambda}, we obtain $a^\theta \coloneq (\theta a')^1$ and $X^\theta \coloneq X^1(\theta a')$  satisfying
    \begin{align} &\Lambda( a^\theta_{4M_\theta};\rho) \le C \Lambda( a'_{\theta,M_\theta} ;c^*\rho), \quad a^\theta \le \mathbbm 1_{X^\theta} , \quad \delta \mathbbm 1_{X^\theta_{\delta'}} \le \theta a',\\
    & \label{e:4mCovering}4M_\theta \ep_\theta + (\ep_\theta \lambda)^{-1} < \rho. \end{align} 
Each $a^\theta$ and $X^\theta$ also depends on $\lambda$ but we suppress the dependence in our notation. The final step is to glue these functions $\theta^{-1} a^\theta$ back together in a smooth fashion to create $a^2$. First, set
    \[ X^2 = \bigcup_{\substack{\theta \in \Theta, \lambda \ge \Lambda(a';\rho/4)}} \theta^{-1} X_{\delta'}^\theta.\]
Clearly, $\delta \mathbbm 1_{X^2_{\delta'}} \le a'$. Next, set
    \[ \tilde a = \sup_{\theta,\lambda} \theta^{-1}a^\theta.\]
From the definition of $\Lambda$ and the construction of $\tilde{a}$, we have 
    \[ \Lambda((\tilde a)_{\theta,4M_\theta}; \rho) \leq \Lambda(a^\theta_{4M_\theta};\rho) \le C \Lambda(a';c^*\rho).\] 
This, together with \eqref{e:4mCovering} shows 
    \[ \Lambda(\tilde a;\rho) \le \max\{ C\Lambda(a';c^* \rho),\Lambda(a';\rho/4)\} \le C \Lambda(a';c\rho).\] 
Furthermore $\tilde a \le \mathbbm 1_{X^2}$, however, $\tilde{a}$ may not be smooth. This can be fixed rather easily since we are not worried about bounding the transfer function of $a^2$, only $a^1$. By Urysohn's Lemma, we can find a $C^\infty$ function $\tilde a \le a_2 \le \mathbbm 1_{X^2}$ since $X^2$ is a $\delta'$ neighborhood of some other set $\tilde X$, and moreover the derivatives of $a^2$ must obey \eqref{e:a-smooth-bounded-2} and since $\Lambda$ is decreasing in the first slot $\Lambda(a^2; \rho) \leq \Lambda(\tilde a;\rho) \leq C \Lambda(a';\rho)$.

\end{proof}

\appendix
\section{Resolvent Estimate to Observability Estimate}\label{s:resolventConnection}
In this appendix we state and prove two results that connect observability resolvent estimates and observability estimates. Because we are interested in observability of the Schr\"odinger equation with the standard and fractional Laplacian, 
we state these results in the abstract Hilbert space setting.

Let $X$ and $Y$ be Hilbert spaces. Let $\Ac: D(\Ac) \ra X$ be a self-adjoint operator. Equivalently, $i\Ac$ generates a strongly continuous semigroup $e^{it \Ac}$ of unitary operators on $X$. Let $X_1$ denote $D(\Ac)$ with the norm 
\begin{equation}
    \|x\|_1 = \|\Ac x\|_X + \|x\|_X.
\end{equation}
Let $\mathscr{C} \in \Lc(X_1,Y)$. We assume that $\mathscr{C}$ is admissible, in the sense that for all $T>0$, there exists $K_T>0$ such that 
\begin{equation}\label{eq:abstractAdmiss}
\int_0^T \|\mathscr{C} e^{it\Ac} x_0\|_Y^2 \leq K_T \|x_0\|_X^2.
\end{equation}
Then $x_0 \mapsto \mathscr{C}e^{it\Ac}(x_0)$ from $D(\Ac)$ to $L^2_{loc}(\Rb; Y)$ has a continuous extension to $X$.

The system
\begin{equation}\label{eq:abstractSchro}
    \dot{x}(t) - i \Ac x(t) = 0, \quad x(0)=x_0\in X, 
\end{equation}
is exactly observable in time $T$ at cost $\kappa_T$ if for all $x_0 \in X$
\begin{equation}\label{eq:abstractObs}
    \|x_0\|_X^2 \leq \kappa_T \int_0^T \|\mathscr{C} e^{it\Ac} x_0\|_Y^2 \,dt. 
\end{equation}
We define the observability resolvent estimate: there exists $M,m>0$ such that for all $x \in D(\Ac)$ and $\lambda \in\Rb$
\begin{equation}\label{eq:abstractResolvent}
    \|x\|^2_X \leq M \|(\Ac-\lambda)x\|^2_X + m \|\mathscr{C} x\|^2_Y.
\end{equation}
We will take $X=Y=L^2(\Rb^d)$, and $\Ac=-\Delta$, $(-\Delta)^{\alpha}$ or $(-\Delta)^{\alpha} +V$ for $\alpha \in [1/2,1]$ and $V \in L^{\infty}(\Rb^d)$. Then $D(\Ac)=H^2$ or $H^{2\alpha}$ respectively. We always take $\mathscr{C}$ to be multiplication by a positive $L^{\infty}$ function; either $a^{1/2} \leq 1$ or $\mathbbm{1}_E$. Because of this, and the unitary nature of $e^{it\Ac}$, the admissibility condition \eqref{eq:abstractAdmiss} is always satisfied with $K_T=T$.

We strongly relied on the equivalence between observability of the Schr\"odinger-type equations and observability resolvent estimates in our proofs of Theorems \ref{thm:gcc} and \ref{thm:comb} above. The standard result in this connection is due originally to L. Miller in \cite[Theorem 5.1]{Mil05}, which states that \eqref{eq:abstractResolvent} and \eqref{eq:abstractObs} for some large time $T$ are equivalent.  
\begin{theorem}[Miller]\label{thm:Miller}
   The system \eqref{eq:abstractSchro} is exactly observable if and only if the observability resolvent estimate \eqref{eq:abstractResolvent} holds. More precisely, for all $\ep>0$, there exists $C_{\ep}>0$ such that \eqref{eq:abstractResolvent} implies \eqref{eq:abstractObs} for all $T>\sqrt{M(\pi^2+\ep)}$ with $\kappa_T = \frac{C_{\ep} m T}{T^2-M(\pi^2+\ep)}$.
\end{theorem}

The main tool we develop here is a method for obtaining arbitrary-time observability under the assumption that $M$ in \eqref{eq:abstractResolvent} depends on $\lambda$, and in fact decays at a rate $\lambda^{-\ep}$ for some $\ep>0$. In the case when $A$ has compact resolvent (e.g. $-\Delta_M$ on a compact manifold $M$) such a result is well-known \cite{BZ04,Mil12}, since one can observe high frequencies arbitrarily fast, and the lower frequencies can be absorbed by compactness.

\begin{theorem}\label{thm:resolventArbObs}
    Suppose 
    \begin{itemize}
        \item $\Ac$ is a semibounded, self-adjoint operator,
        \item $\mathscr C$ is bounded,
        \item there exists $M,m>0$ and $\ep \in (0,1]$ such that
            \begin{equation}\label{e:resEp}\norm{f}_X^2 \le M (1+|\lambda|)^{-\ep} \norm{(\Ac-\lambda)f}_X^2 + m \norm{\mathscr C f}_Y^2\end{equation}
        for all $\lambda \in \mathbb R$ and $f \in \mathcal D(\Ac)$.
    \end{itemize}
Then, there exists $C>0$ such that for all $T\in (0,1)$ and all $x_0 \in X$,
    \[ \norm{x_0}_X^2 \le C e^{C{T^{2-4/\ep}}} \int_0^T \norm{\mathscr C e^{it\Ac}x_0}_Y^2 \, dt.\]
\end{theorem}
There are two main ingredients in proving Theorem \ref{thm:resolventArbObs}. The first is the main result of \cite[Theorem 1]{DM12} which states that the decaying resolvent estimate \eqref{e:resEp} suffices for final-time heat observability. 
\begin{theorem}[Duyckaerts-Miller]\label{thm:DMheat}
Under the same assumptions as Theorem \ref{thm:resolventArbObs}, there exists $C>0$ such that for all $T>0$ and all $x_0 \in X$,
    \[ \norm{e^{-T\Ac}x_0} \le C e^{CT^{1-2/\ep}} \int_0^T \norm{\mathscr C e^{-t\Ac}x_0}_Y^2 \, dt.\]
\end{theorem}
One small technicality is that \cite[Theorem 1]{DM12} is only stated for $\Ac$ which are positive definite. However, by applying that result to $\Ac+RI$ for $R>0$ large, we can obtain Theorem \ref{thm:DMheat}, modifying $C$ appropriately, in terms of $R$; see the reduction in the beginning of the proof of Theorem \ref{thm:resolventArbObsplusHeat} in Section \ref{s:appMainProof} below.

The second main ingredient in the proof of Theorem \ref{thm:resolventArbObs} is an abstract version of the argument in \cite{SSY25}, where it is shown that the decaying resolvent estimates plus arbitrary time heat observability yields arbitrary time Sch\"odinger observability.
\begin{theorem}\label{thm:resolventArbObsplusHeat}
    Assume that $\Ac$ has semi-bounded spectrum, $\sigma(\Ac)\subset [-R,\infty)$, and $\mathscr{C} \in \Lc(X,Y)$.
    Assume there exist $m, \lambda_0>0$ and $M(\lambda)$ a decreasing function satisfying $\lim\limits_{\lambda \ra \infty} M(\lambda)=0$ such that 
    \begin{equation}\label{eq:abstractObsResolveDecay}
        \|f\|_X^2 \leq M(\lambda) \|(\Ac-\lambda) f \|^2_X + m \|\mathscr{C} f\|_Y^2,
    \end{equation}
    for all $f \in D(\Ac)$ and $\lambda \geq \lambda_0$. Assume further that there exists $r:(0,1) \to (0,\infty)$ a decreasing function such that 
    \begin{equation}\label{eq:abstractHeatObs}
        \|e^{-T \Ac} u_0\|_X^2 \leq r(T) \int_0^T \| \mathscr{C} e^{-t \Ac} u_0\|_Y^2 dt,
    \end{equation}
    for all $u_0 \in X$ and $T \in (0,1)$. Then there exists $C>0$ such that for all $T\in(0,1)$ we have the observability estimate 
    \begin{equation}\label{eq:abstractSchrodingerArbObs}
        \|u_0\|_X^2 \leq C(T,r,R,M,m) \int_0^T \|\mathscr{C} e^{it \Ac} u_0 \|^2_Y dt 
    \end{equation}
    for all $u_0 \in X$. The precise form of the constant in \eqref{eq:abstractSchrodingerArbObs} is given below in \eqref{e:preciseConstantHeatObs}.
\end{theorem}

Clearly Theorem \ref{thm:resolventArbObs} is an immediate consequence of Theorems \ref{thm:resolventArbObsplusHeat} and \ref{thm:DMheat}. The latter is already proved, so the remainder of this appendix is mostly dedicated to proving Theorem \ref{thm:resolventArbObsplusHeat} in Section \ref{s:appMainProof} below.

Before moving to that, we point out that if we have a high frequency observability resolvent estimate for the fractional Schr\"odinger equation, then relative density is sufficient to prove an observability resolvent estimate for the low frequencies. In other words, estimates like \eqref{e:resEp} automatically extend from high frequencies to low when $\Ac = (-\Delta)^\alpha$, $\alpha>0$. 
\begin{lemma}\label{l:lowFreqResolve}
    Suppose $\alpha>0$ and $a:\mathbb R^d \to [0,1]$ is relatively dense. If there exists $M(\lambda),m, \lambda_0 >0$ such that for all $u \in H^{2\alpha}(\Rb^d)$ and $\lambda \geq \lambda_0$ we have 
    \begin{equation}
        \ltwo{u}^2 \leq M(\lambda) \ltwo{\left[ (-\Delta)^{\alpha}-\lambda\right]u}^2 + m \ip{au}{u},
    \end{equation}
    then there exists $C>0$ such that for all $u \in H^{2\alpha}(\Rb^d)$ and all $\lambda \in \Rb$ we have 
    \begin{align}
        &\ltwo{u}^2 \leq M^*(\lambda) \ltwo{\left[ (-\Delta)^{\alpha}-\lambda\right]u}^2 +  C\ip{au}{u},\\
        &M^*(\lambda) = \left\{ \begin{array}{ll} |\lambda|^{-2}, & \lambda \le -1; \\ C, & -1 < \lambda < \lambda_0; \\ M(\lambda), & \lambda \ge \lambda_0. \end{array} \right. 
    \end{align}
\end{lemma}
\begin{proof}
    Due to positivity of $(-\Delta)^\alpha$ and the assumed resolvent bound, it suffices to show the desired resolvent estimate for $-1 < \lambda <\lambda_0$. Notice that if $\hat g = \mathbbm 1_{B(0,\mu)} \hat u$ with $\mu=(\lambda_0+2)^{\frac 1{2\alpha}}$ then by the PLS Theorem (Theorem \ref{thm:pls}) we obtain the existence of $C$ depending on $a$, $\alpha$, and $\lambda_0$ such that
        \[ \norm{g}_{L^2} \le C \norm{a^{1/2}g}_{L^2}.\]
    On the other hand, if $\xi$ is outside $B(0,\mu)$ then
        \[ ||\xi|^{2\alpha}-\lambda| \ge \mu^{2\alpha} -|\lambda| \ge 1.\] 
    And therefore, setting $\hat f = \mathbbm 1_{B(0,\mu)^c}\hat u$ then
        \[ \norm{f}^2_{L^2} = \int_{B(0,\mu)^c} \abs{\hat u}^2 \le \norm{\left[ (-\Delta)^\alpha -\lambda\right]u }_{L^2}^2. \]
    Finally, by the triangle inequality, as in \eqref{e:triangle},
        \begin{align} \ltwo{u} &\le \ltwo{f} + \ltwo{g} \le C\ltwo{a^{1/2} u} + (C+1)\ltwo{f} \\ 
        &\le C\ltwo{a^{1/2} u} + (C+1)\norm{\left[ (-\Delta)^\alpha -\lambda\right]u }_{L^2}.\end{align}
    
\end{proof}

\subsection{Proof of Theorem \ref{thm:resolventArbObsplusHeat}}\label{s:appMainProof}
First, note that it suffices to prove Theorem \ref{thm:resolventArbObsplusHeat} in the case $\Ac \ge 1$ since we can replace $\Ac$ by $\Ac + (R + 1)I$. Indeed, the resolvent estimate \eqref{eq:abstractObsResolveDecay} implies
    \begin{align} &\norm{f}_X^2 \le M^*(\lambda) \norm{[\Ac + (R+1)I - \lambda]f }_X^2 + m \norm{\mathscr{C} f}_Y^2, \\ 
    &M^*(\lambda) = M(\lambda-R-1), \quad \lambda \ge \lambda_0 + R+1. \end{align} 
Furthermore, writing $e^{-t\Ac} = e^{t(R+1)} e^{-t(\Ac +(R+1)I)}$, the heat observability inequality \eqref{eq:abstractHeatObs} implies 
\begin{equation}\norm{e^{-T(\Ac + (R+1)I)}u_0}_X \le r(T)e^{(R+1)T} \int_0^T \ \norm{\mathscr C e^{-t(\Ac + (R+1)I)}u_0 }_Y^2 \, dt,\end{equation} 
for all $u_0 \in X$.
Therefore, the conditions of Theorem \ref{thm:resolventArbObsplusHeat} are satisfied with a positive operator and $r^*(T) = e^{R+1} r(T)$, and hence we obtain $C^*>0$ such that for all $u_0 \in X$,
    \begin{equation}
        \|u_0\|_X^2 \leq 
        C^*(T,r^*,-1,M^*,m)\int_0^T \|\mathscr{C} e^{it(\Ac+(R+1)I)} u_0\|^2 dt.
    \end{equation}
However, rewriting $\mathscr{C} e^{it(\Ac+(R+1)I)} = e^{it(R+1)} \mathscr C e^{it\Ac}$ and 
    \[ C(T,r,R,M,m) = C^*(T,r^*,-1,M^*,m),\] 
yields \eqref{eq:abstractSchrodingerArbObs}. We now turn to proving Theorem \ref{thm:resolventArbObsplusHeat} when $\Ac \ge 1$. The first ingredient toward this end is extending the heat observability \eqref{eq:abstractHeatObs} to solutions of the following inhomogeneous heat equation,
\begin{equation}\label{eq:abstractInhomogHeat}
    (\p_t + \Acr) u = F \in L^2((0,\infty): X), \quad u|_{t=0} =u_0 \in D(\Ac).
\end{equation}
\begin{lemma}\label{l:inhomHeatObs}
    Assume that 
    \begin{itemize}
        \item $\Ac \ge 1$ 
        \item $\|\mathscr{C}\|_{\Lc(X,Y)} \leq \ti{C}$,
        \item there exists $r:[0,\infty) \to (0,\infty)$ 
    such that 
    \begin{equation}\label{eq:abstractHeatObs2}
        \|e^{-T \Ac} u_0\|_X^2 \leq r(T) \int_0^T \| \mathscr{C} e^{-t \Ac} u_0\|_Y^2 dt,
    \end{equation}
    for all $u_0 \in X$ and $T>0$. 
    \end{itemize}
    Then there exists $C>0$ such that for any solution $u$ of \eqref{eq:abstractInhomogHeat} we have 
\begin{equation}
    \|u(T)\|_X^2 \leq 12 \ti{C} r(T) \int_0^T \left( \|\mathscr{C}\Acr u(t)\|^2_Y + \|F(t)\|_X^2 \right) dt. 
\end{equation}
\end{lemma}
\begin{proof}
    We decompose the solution into $u(t)=u_1(t)+u_2(t)$ where $u_1, u_2$ solve the homogeneous problem with initial data $u_0$, and the inhomogeneous problem with trivial initial data, respectively. That is 
    \begin{equation}
        \begin{cases}
            \p_t u_1 + \Acr u_1 = 0, \quad u_1|_{t=0}=u_0 \\
            \p_t u_2 + \Acr u_2 = F, \quad u_2|_{t=0}=0.
        \end{cases}
    \end{equation}
    Since $\Acr \geq 1$ and using the assumed heat observability \eqref{eq:abstractHeatObs2} we have 
    \begin{equation}
        \|u_1(T)\|_X^2 \leq \|\Acr u_1(T)\|_X^2 \leq r(T) \int_0^T \|\mathscr{C}\Acr u_1(t)\|^2_Y dt.
    \end{equation}
    We can rewrite $\Acr u_1= \Acr u + \p_t u_2 - F$, and use that $\mathscr{C}: Y \ra X$ is bounded to obtain
    \begin{equation}\label{eq:abstractHeatImhomogIntermed1}
        \|u_1(T)\|_X^2 \leq 3r(T)\int_0^T \|\mathscr{C}\Acr u(t)\|^2_Y dt + 3 \ti{C} r(T) \int_0^T \left( \|\p_t u_2(t)\|_X^2 + \|F(t)\|_X^2 \right) dt.
    \end{equation}
    Now to estimate $u_2$ terms, we pair the equation for $u_2$ with $\p_t u_2$ in $X$ then take real parts to obtain
    \begin{equation}
        \|\p_t u_2\|_X^2 + \Re \<(\Ac+R)u_2, \p_t u_2\>_X = \Re \<F, \p_t u_2\>_X.
    \end{equation}
    We can rewrite the second term 
    \begin{equation}
        \|\p_t u_2\|_X^2 + \frac{1}{2}\p_t\<(\Ac+R)u_2,  u_2\>_X = \Re \<F, \p_t u_2\>_X.
    \end{equation}
    Now integrating in $t$ from $0$ to $T$, using that $(\Ac+R) \geq 1$ and $u_2|_{t=0}$, we obtain
    \begin{equation}
        \int_0^T \|\p_t u_2(t)\|_X^2 dt + \frac 12 \|u_2(T)\|_X^2 \leq \int_0^T \Re \<F(t),\p_t u_2(t)\>_X dt.
    \end{equation}

    Estimating $2|\ip{F(t)}{\p_t u_2(t)}| = \norm{F(t)}^2_X + \norm{\p_tu_2(t)}_X^2$, yields \begin{equation}\label{eq:abstractHeatImhomogIntermed2}
        \|u_2(T)\|_X^2 + \int_0^T \|\p_t u_2(t)\|_X^2 dt \leq \int_0^T \|F(t)\|^2_X dt.
    \end{equation}
    Finally combining \eqref{eq:abstractHeatImhomogIntermed1} and \eqref{eq:abstractHeatImhomogIntermed2} with Cauchy-Schwarz we have 
    \begin{align}
        \|u(T)\|_X^2 &\leq 2 \left( \|u_1(T)\|_X^2 + \|u_2(T)\|_X^2 \right) \\
        &\leq 12 \ti{C} r(T) \int_0^T \left( \|\mathscr{C}(\Ac+R) u(t)\|_Y^2  + \|F(t)\|_X^2 \right) dt,
    \end{align}
    which is the desired inequality.
\end{proof}

Consider the abstract inhomogeneous backward heat equation: 
\begin{equation}\label{eq:backwardHeat}
    \p_t u -\Acr u = F \in L^{2}((0,\infty); X), \quad u|_{t=T} = u_T \in D(\Ac).
\end{equation}
By reversing time $t$ to $T-t$ we obtain a solution to \eqref{eq:abstractInhomogHeat} and hence we can extend Lemma \ref{l:inhomHeatObs} in the following way.
\begin{lemma}\label{l:backwardHeatObs}
Under the assumptions of Lemma \ref{l:inhomHeatObs}, there exists $C>0$ such that for any $T>0$ and any $u$ solving \eqref{eq:backwardHeat} we have 
\begin{equation}
\|u(0)\|_X^2 \leq 12 \ti{C} r(T) \int_0^T \left( \| \mathscr{C}\Acr u(t)\|^2_Y + \|F(t)\|_X^2 \right) dt.
\end{equation}
\end{lemma}
Next we can convert the inhomogeneous backwards heat observability result to an observability result for the Schr\"odinger equation, with additional factors of $\Ac$.
\begin{lemma}\label{l:schroArbObsWeak}
Assume that $\Ac \ge 1$ 
and $\|\mathscr{C}\|_{\Lc(X,Y)} \leq \ti{C} $.
    Assume further that there exists  
    $r:(0,\infty) \to (0,\infty)$ such that
    \begin{equation}\label{eq:abstractHeatObs3}
        \|e^{-T \Ac} u_0\|_X^2 \leq r(T) \int_0^T \| \mathscr{C} e^{-t \Ac} u_0\|_Y^2 dt,
    \end{equation}
    for all $u_0 \in X$ and $T>0$. 
    Then there exists an absolute constant $C \le 10^4 \ti{C}$ and such that for any $T \in (0,1)$, $h \in \left(0, \frac{T^2}{C\ln(C r(T/10)T^{-1})}\right)$, and $f \in D(\Ac)$ we have 
    \begin{equation}\label{e:SchroObswithA}
        \|f\|_X^2 \leq {Ch} \|\Acr f\|_X^2 + C \frac{T^2}{h}e^{\frac{T^2}{h}}r(T/10) \int_0^T \|\mathscr{C}\Acr e^{it\Acr }f\|_Y^2 dt. 
    \end{equation}
\end{lemma}
\begin{proof}
Following \cite{SSY25} we introduce the following FBI-type transform $\mathcal T_h$. Let $g$ be a normalized Gaussian function
    \[ g(z) = \frac{1}{\sqrt{2\pi}}e^{-\frac{z^2}{2}}, \quad g_h(z) = h^{-1/2} g(zh^{-1/2}), \quad 0 < h < 1.\] 
The key property is that $\int_{\mathbb R} g_h = 1$ for all $h$.
For $0<h<1$, $z = \tau+is \in \mathbb{C}$ and $\Gamma(t)$ a bounded $X$-valued function, we define 
\begin{equation}
    \mathcal{T}_h \Gamma(z) = \frac{1}{\sqrt{2\pi h}} \int_{\Rb} e^{-\frac{(z+t)^2}{2h}} \Gamma(t) dt = \int_{\mathbb R} g_h(z+t) \Gamma(t) \, dt.
\end{equation}
Differentiating the kernel and integrating by parts reveals the crucial property, $\partial_s \mathcal {T}_h = -i\mathcal{T}_h \partial_t$. Therefore $\mathcal {T}_h$ will allow us to connect the backward heat equation \eqref{eq:backwardHeat} and the Schr\"odinger equation \eqref{eq:abstractSchro}.

Now, fix $T>0$ and choose $\chi \in C^1([0,10T]:[0,1])$ with 
\begin{equation}
    \chi\equiv 1 \text{ on } [2T, 8T], \quad |\chi'| \leq \frac{2}{T}.
\end{equation}
For $f \in D(\Ac)$ we denote $F=e^{it\Acr} f$ and $\ti{F}=\chi F$. Then 
\begin{equation}
    i \p_t \ti{F} + \Acr \ti{F} = i \chi'(t) F.
\end{equation}
Setting $W=\mathcal{T}_h (\ti{F})$ and $G=-\mathcal{T}_h(i \chi' F)$ we obtain
\begin{equation}
    \p_s W - \Acr W = G,
\end{equation}
i.e. for each $\tau \in \mathbb R$, $s \mapsto W(\tau + is)$ satisfies the inhomogeneous backward heat equation \eqref{eq:backwardHeat}. The proof will now consist of two main steps. First, we will observe $\norm{W(\tau)}_X$ by $\mathscr{C}\Acr F$ for suitable $\tau$ using Lemma \ref{l:backwardHeatObs}. Then, we will control $\norm{W(\tau)}_X$ from below by $\norm{f}$.

\subsection*{Observing $W(\tau)$}
By Lemma \ref{l:backwardHeatObs}, for any $\tau \in \Rb$ we have 
\begin{equation}\label{eq:interMedHeat}
    \|W(\tau)\|_X^2 \leq 12 \ti{C} r(T) \int_0^T \left( \|\mathscr{C}\Acr W(\tau+is) \|^2_Y + \|G(\tau+is)\|_X^2 \right) ds.
\end{equation}
Some elementary estimates of the kernel of $\mathcal {T}_h$ will we used to further estimate each of the terms of the right-hand side.
From the definition of the FBI transformation and $W(\tau+is)$ we have 
\begin{equation}
    \mathscr C \Acr W(\tau+is) = \frac{1}{\sqrt{2\pi h}} e^{\frac{s^2}{2h}} \int_{0}^{10T} e^{-\frac{(\tau+t)^2}{2h}} e^{-i\frac{s(\tau+t)}{h}} \chi(t) \mathscr C\Acr F(t) dt.
\end{equation}
Taking the $Y$ norm of both sides, then exchanging the order of the norm and integral, and applying Cauchy-Schwarz in the $t$ integral we obtain
\begin{equation}\label{eq:abstractAWinequality}
    \|\mathscr{C}\Acr  W(\tau+is)\|^2_Y \leq \frac{10
    T}{2 \pi h} e^{\frac{s^2}{h}} \int_0^{10T} \|\mathscr{C} \Acr F(t)\|^2_Y dt.
\end{equation}
On the other hand, by the definition of the FBI transform
\begin{equation}
    G(\tau+is) = -i \frac{1}{\sqrt{2\pi h}} e^{\frac{s^2}{2h}} \int_{\Rb} e^{-\frac{(\tau+t)^2}{2h}} e^{-i \frac{s(\tau+t)}{h}} \chi'(t) F(t) dt.
\end{equation}
Since $\chi'=0$ outside of the intervals $[0,2T]$ and $[8T,10T]$, for $\tau \in [-6T,-4T]$ and $t$ in the support of $\chi'$, we have $|\tau+t| \ge 4T$.

Combining this with the fact that $e^{it\Acr}$ is unitary on $X$, 
\begin{equation}\label{eq:abstractGinequality}
    \max_{\tau \in [-6T,-4T]} \|G(\tau+is)\|_X^2 \leq \frac{64}{2\pi h} e^{\frac{s^2}{h}} e^{-\frac{4T^2}{h}} \|f\|_X^2.
\end{equation}
Now combining \eqref{eq:abstractAWinequality} and \eqref{eq:abstractGinequality} with \eqref{eq:interMedHeat}, we obtain 
\begin{align} \label{eq:abstractWIntermed}
    \max_{\tau \in [-6T,-4T]} \|W(\tau)\|_X^2 \leq &\frac{384 \ti{C}T}{\pi h} e^{-\frac{3T^2}{h}} r(T) \|f\|_X^2 \\
    &+\frac{30 \ti{C} T^2}{\pi h} e^{\frac{T^2}{h}}r(T) \int_0^{10T} \|\mathscr{C}\Acr F(t)\|^2_Y dt.
\end{align}
\subsection*{Controlling $\norm{W(\tau)}_X$ from below}
Plugging $s=0$ to the definition of the FBI transform, and recalling $W = \mathcal T_h \tilde F$, we can express $W(-\tau) = \tilde F * g_h(\tau)$. 

Therefore, as $h \to 0$, we expect $W(-\tau) \to \tilde F(\tau)$, which in turn equals $F(\tau)$ for $\tau \in [4T,6T]$. To bound $W(\tau)$ from below, we will precisely control the error in this approximation; see \eqref{e:approxError} below. Using the fact that $\int g_h=1$,
    \begin{align} W(-\tau) - \tilde F(\tau) &= \int \left[ \tilde F(t) - \tilde F(\tau) \right] g_h(\tau-t) \, dt. \end{align}
We estimate the difference $\|\tilde F(t)-\tilde F(\tau)\|_X$ using the Mean Value Theorem and the observation that
    \begin{align} \norm{\partial_t \tilde F(t)}_X &\le \norm{\partial_t F(t)}_X + \norm{F(t)}_X |\chi'(t)| \\
    &\le \norm{\Acr f}_X + \frac{2}{T}\norm{f}_X .
    \end{align}
In this way,
    \[ \int_{\mathbb R} \left\| \tilde F(t) - \tilde F(\tau) \right\|_X |g_h(\tau-t)| \, dt \le \left( \norm{\Acr f}_X {+ \frac{2}{T}\norm{f}_X}\right) \int_{\mathbb R} |t| g_h(t) \, dt .\]
However the final integral equals $\frac{2\sqrt{h}}{\sqrt{2\pi}}$.
Therefore, we have shown for any $\tau \in \mathbb R$,
    \begin{equation}\label{e:approxError}\norm{W(-\tau)-\tilde F(\tau)}_X \le \frac{2 \sqrt{h}}{\sqrt{2\pi}} \left(\norm{\Acr f}_X {+ \frac{2}{T} \norm{f}_X} \right).\end{equation}
\subsection*{Conclusion}
Now, for any $\tau \in [4T,6T]$, say $\tau=5T$, we apply \eqref{eq:abstractWIntermed} and \eqref{e:approxError} to obtain
    \begin{align} \frac{1}{2}\norm{f}_X^2 &= \frac{1}{2}\norm{\tilde F(\tau)}_X^2 \\
    & \le \norm{W(-\tau)}^2 + \norm{W(-\tau) - \tilde F(\tau)}_X^2 \\
    & \le \frac{4h}{\pi} \norm{\Acr f}_X^2 + \frac{30 \ti{C}T^2}{\pi h} e^{\frac{T^2}{h}}r(T) \int_0^{10T} \|\mathscr{C}\Acr F(t)\|^2_Y dt \\
    & \qquad + \left( \frac{384 \ti{C} T}{\pi h} e^{-\frac{3T^2}{h}} r(T) {+ \frac{16h}{\pi T^2} }\right)\|f\|_X^2.\end{align}
Finally, we take $h$ small enough that the final term can be absorbed in the left-hand side. Indeed, if 
    \[ h \le h_0 \coloneq \frac{2T^2}{\ln\left( \frac{{8} r(T) 384 \ti{C}}{ \pi T} \right)} {\frac{\pi}{8\cdot 16}}\]
then {$\frac{16h}{\pi T^2} \leq \frac{1}{8}$}. Furthermore, 
$e^{-\frac{2T^2}{h}} \le \frac{\pi T}{8r(T) \cdot 384\ti{C}}$, and hence
    \[ \frac{384\ti{C}T}{\pi h} e^{-\frac{3T^2}{h}} r(T) = \frac{384\ti{C} r(T)}{ \pi T} \left[ \frac{T^2}{h} e^{-\frac{3T^2}{h}} \right] \le \frac{384\ti{C} r(T)}{ \pi T}e^{-\frac{2T^2}{h}} \le \frac{1}{8}. \]
Summarizing with less precise constants, we obtain
    \[ \norm{f}_X^2 \le 40 \left( {h} \norm{\Acr f}_X + \frac{\ti{C} T^2}{h} e^{\frac {T^2}{h}} r(T) \int_0^{10T} \norm{\mathscr{C} \Acr e^{-i\Acr t} f}_Y \, dt \right), \quad 0< h < h_0.\]
Finally, replacing $T$ by $T/10$ (since $T$ was arbitrary) and similarly adjusting $h$ yields \eqref{e:SchroObswithA}.

\end{proof}
Since $\Acr:D(\Ac) \ra X$ is self-adjoint with $\sigma(A) \subset [1,\infty)$ 
it has a spectral measure $dE_{\rho}$ such that  
\begin{equation}
    \Acr u = \int_{1}^{\infty} \rho \, dE_{\rho}(u),
\end{equation}
Then, define the spectral projection operator as  
\begin{equation}
    \Pi_{\lambda }f = \int_1^{\lambda} d E_{\rho}(f).
\end{equation}
We now use the spectral projector, to show that the assumed resolvent estimate implies a ``high frequency" Schr\"odinger observability estimate. Because the proofs are similar, we also prove that observability resolvent estimates \eqref{eq:abstractResolvent} imply observability \eqref{eq:abstractObs}. 

\begin{lemma}\label{l:HighFreqArbTimeObs}
Assume that $\sigma(\Ac)\subset [1,\infty)$, and $\mathscr{C}$ is admissible. 
\begin{enumerate} 
\item 
    Assume there exist $m, \lambda_0>0$ and $M(\lambda)$ a decreasing function satisfying $\lim\limits_{\lambda \ra \infty} M(\lambda)=0$ such that 
    \begin{equation}
        \|f\|_X \leq M(\lambda) \|(\Ac-\lambda) f \|^2_X + m \|\mathscr{C} f\|_Y^2,
    \end{equation}
    for all $f \in D(\Ac)$ and $\lambda \geq \lambda_0$. 

Then, for any $T>0$, and
    \[ \lambda > \lambda_1 \coloneq \max\left\{ \lambda_0, 2M^{-1}\left(\tfrac{T^2}{200} \right), \tfrac{40}{T}\right\},\]
and any $f \in X$ there holds
\begin{equation}
    \|(I-\Pi_{\lambda})f\|_X^2 \leq \frac{4m}{T} \int_0^T
 \|\mathscr{C} e^{it\pAcr}(I-\Pi_{\lambda}) f \|^2_Y dt.
\end{equation}

\item Assume that there exist $M,m>0$ such that for all $x \in D(A)$ and $\lambda \in \Rb$
\begin{equation}
    \|x\|_X^2 \leq M \|(\Ac-\lambda)x\|_X^2 + m \|\mathscr{C} x\|_Y^2.
\end{equation}
Then there exists $C,T>0$ such that for all $x_0 \in X$
\begin{equation}
    \|x_0\|_X^2 \leq C \int_0^T \|\mathscr{C} e^{it\Ac} x_0\|_Y^2 dt.
\end{equation}
\end{enumerate}
\end{lemma}
\begin{proof}
    Let $f \in D(\Ac)$ and let $F$ be the solution of the Schr\"odinger equation 
    \begin{equation}
        i\p_tF +\Ac F = 0, \quad F|_{t=0} \in D(\Ac). 
    \end{equation}
    Fix a cutoff function $\chi \in C^2([0,T]:[0,1])$ such that $\chi \equiv 1$ on $[T/4,3T/4]$ and $|\chi'| \leq 5/T$. Then for $T>0$ consider the $X$ valued function $ \Psi(t) = \chi\left(t\right) F(t)$
    which solves 
    \begin{equation}
        i\p_t \Psi +\Acr \Psi = i \chi'F.
    \end{equation}
    Then the Fourier transform of $\Psi$ with respect to $t$ satisfies for each $\tau \in \mathbb R$,
    \begin{equation}\label{eq:highFreqObsFourierDef}
        (\Acr-\tau) \hat{\Psi}(\tau) = i\Fc\left( \chi' F\right)(\tau).
    \end{equation}
    \subsubsection*{Proof of (1)} Let $F|_{t=0}= (1-\Pi_{\lambda}) f \in D(\Ac)$. First take $\lambda \geq 2\lambda_0$. Then by \eqref{eq:abstractObsResolveDecay} for $\tau>\frac 12 \lambda \ge \lambda_0$ we have 
    \begin{align}
        \|\hat{\Psi}(\tau)\|_X^2 &\leq M(\tau)\|(\Acr-\tau)\hat{\Psi}(\tau)\|_X^2 + m \|\mathscr{C}\hat{\Psi}(\tau)\|^2_Y\\
        &\leq M(\lambda/2)\|\Fc(\chi' F)(\tau)\|_X^2 + m \|\mathscr{C} \hat{\Psi}(\tau)\|^2_Y. \label{eq:highFreqObsIntermed1}
    \end{align}
    On the other hand, since $\Pi_\lambda \hat \Psi(\tau)=0$, using the spectral measure we have 
    \begin{equation}
        (\Acr-\tau) \hat{\Psi}(\tau) = \int_{\lambda}^{\infty}(\rho-\tau) dE_{\rho} \hat{\Psi}(\tau).
    \end{equation}
    Now note that for $\rho \geq \lambda$ and $\tau \le \frac 12 \lambda$, $\frac{\lambda}{2}<\rho - \tau$. 
    Therefore, when $\tau \leq \frac 12 \lambda$, 
    \begin{align}
        \frac{\lambda}{2} \|\hat{\Psi}(\tau)\|^2 &\leq \int_{\lambda}^{\infty}(\rho - \tau) (dE_{\rho} \hat{\Psi}(\tau), \hat{\Psi}(\tau))_X \\
        &= \langle(\Acr-\tau) \hat{\Psi}(\tau), \hat{\Psi}(\tau) \rangle_X.
    \end{align}
    Then using Cauchy-Schwarz and applying \eqref{eq:highFreqObsFourierDef}, for $\tau \leq \frac{\lambda}{2}$ we have 
    \begin{equation}\label{eq:highFreqObsIntermed2}
        \|\hat{\Psi}(\tau)\|_X^2 \leq \frac{4}{\lambda^2} \|\Fc(\chi' F)(\tau)\|_X^2.
    \end{equation}
    Combining \eqref{eq:highFreqObsIntermed1} and \eqref{eq:highFreqObsIntermed2}, then integrating in $\tau$ over $\Rb$ we have
    \begin{equation}
        \int_{\Rb} \|\hat{\Psi}(\tau)\|_X^2 d\tau \leq \left(M(\lambda/2) + \frac{4}{\lambda^2} \right) \int_{\Rb} \|\Fc(\chi' F)(\tau)\|_X^2 d \tau + m \int_{\Rb} \|\mathscr{C} \hat{\Psi}(\tau)\|_Y^2 d \tau. 
    \end{equation}
    Using Plancherel, recalling the form of $\Psi$ and $\chi$, and using the unitarity of $F(t)$,
    \begin{align}
        \frac{T}{2} &\norm{(I-\Pi_\lambda)f}_X^2  \le \int_{\Rb} \|\chi(t) F(t)\|_X^2 d t \\
        &\leq \left(M(\lambda/2) + \frac{4}{\lambda^2} \right) \int_{\Rb} \|\chi'(t) F(t)\|_X^2 dt + m \int_{\Rb} \|\mathscr{C} F(t)\|^2_Y d t \\
        &\leq \left(M(\lambda/2) + \frac{4}{\lambda^2} \right) \frac{25}{T} \|(I-\Pi_\lambda)f\|_X^2 + m \int_{0}^T \|\mathscr{C} e^{it\pAcr}(I-\Pi_\lambda)f\|^2_Y d t.
    \end{align}
    Taking $\lambda$ large enough, we may absorb the error term back into the left hand side to obtain the desired inequality for $f \in D(\Ac)$. More precisely, if $\lambda \ge 2 M^{-1}(T^2/200)$ then $M(\lambda/2)\frac{50}{T^2} \le \frac 14$. Furthermore, if $\lambda \ge 20\sqrt{2}/T$ then $\frac{4}{\lambda^2} \frac{50}{T^2} \le \frac 14$. Therefore, slightly relaxing the constants, if
        \[ \lambda \ge 2 \max\left\{ M^{-1}\left(\tfrac{T^2}{200} \right), \tfrac{20}{T}\right\}, \] 
    then we obtain the desired inequality. A density argument completes the proof for $f \in X$.
     \subsubsection*{Proof of (2)} Let $F|_{t=0}=x_0$ and apply \eqref{eq:highFreqObsFourierDef} and \eqref{eq:abstractResolvent} to obtain for all $\tau \in \Rb$
    \begin{align}
        \|\hat{\Psi}(\tau)\|_X^2 
        &\leq M\|\Fc(\chi' F)(\tau)\|_X^2 + m \|\mathscr{C} \hat{\Psi}(\tau)\|_Y^2.
    \end{align}
    Now following the second half of the Proof of (1), we integrate in $\tau$ over $\Rb$, then use that the Fourier transform is unitary to obtain 
    \begin{equation}
        \|x_0 \|_X^2 \leq \frac{2m}{T}\int_0^T \| \mathscr{C} e^{it\Ac}x_0 \|^2_Y dt + 50  \frac{M}{T^2}\|x_0\|_X^2. 
    \end{equation}
    Taking $T$ large enough, we may absorb the error term back into the left hand side to obtain the desired inequality for $x_0 \in D(\Ac)$. A density argument completes the proof for $x_0 \in X$.
\end{proof}

Now we can conclude the proof of Theorem \ref{thm:resolventArbObs}.
\begin{proof}[Proof of Theorem \ref{thm:resolventArbObs}]
    For any $u_0 \in X$ let $U_0=\pAcr^{-1}u_0 \in D(\Ac)$. Then by Lemma \ref{l:schroArbObsWeak}, and noting that $\pAcr^{-1}$ commutes with $e^{it\pAcr}$, we have 
    \begin{equation}\label{eq:abstractResolvObsWeak1}
        \|U_0\|_X^2 \leq {C h} \|u_0\|_X^2 + C \frac{T^2}{h}e^{\frac{T^2}{h}}r(T/10) \int_0^T \|\mathscr{C}e^{it\pAcr }u_0\|_Y^2 dt.
    \end{equation}
    for any $T \in (0,1)$, $h \in  \left(0, h_0\right)$.

    To work towards a lower bound of the left hand side of 
    \eqref{eq:abstractResolvObsWeak1} we use the high frequency Schr\"odinger observability from Lemma \ref{l:HighFreqArbTimeObs} and the triangle inequality, 
    to obtain 
    \begin{equation}\label{eq:abstractResolvObsIntermed2}
        \|u_0\|_X^2 \leq \frac{8m}{T} \int_0^T \|\mathscr{C} e^{it\pAcr}u_0\|_Y^2 dt + 2 \|\Pi_{\lambda} u_0\|_X^2 
    \end{equation}
    for all $T>0$ and $\lambda > \lambda_1$. 

    Next, from the definition of $U_0$ and $\Pi_{\lambda}$ we have 
    \begin{equation}
        \|\Pi_{\lambda} u_0\|^2 = \int_{0}^{\lambda} (dE_{\rho}u_0,u_0)_X =  \int_0^{\lambda} \rho^2 \rho^{-2} (dE_{\rho}f,f)_X \leq \lambda^2 \|U_0\|^2_X.
    \end{equation}
    This provides the link between \eqref{eq:abstractResolvObsIntermed2} and \eqref{eq:abstractResolvObsWeak1}. Therefore, 
    \begin{equation}
        \|u_0\|_X^2 \leq C\left( \lambda^2 \frac{T^2}{h}e^{\frac{T^2}{h}} r(T/10)+\frac{1}{T}\right) \int_0^T \|\mathscr{C} e^{it\pAcr}u_0\|_Y^2 dt + C \lambda^2 {h} \|u_0\|_X^2.
    \end{equation}
    Taking $\lambda = \lambda_1$ and $h=\ep\min\left({\lambda_1^{-2}}, h_0 \right)$ for $\ep>0$ small enough, we can absorb the second term on the right hand side back. To control the constant with this choice of $h$, we consider all four cases 
        \[ \frac{T^2}{h} \in \left\{ C\lambda_1^2{T^2},C \ln\left( \frac{C r(T/C)}{T}\right) \right\}, \quad \lambda_1 \in \left\{ CM^{-1}\left(\frac{T^2}{C}\right), C/T \right\}.\]
    By considering each one, we obtain
        \begin{equation}\label{e:preciseConstantHeatObs} \lambda^2 e^{CT^2/h} r(T/10) + 1/T \le \exp \left(  C\left( {T} M^{-1}\left( T^2/C \right) \right)^2 \right) \left( r(T/C) + 1\right)^C. \end{equation}
        {Plugging in $M(\lambda)=\lambda^{-\ep}$ and $r(T)=Ce^{CT^{1-2/\ep}}$ gives the control cost of $Ce^{C T^{2-\frac{4}{\ep}}}$ in Theorem \ref{thm:resolventArbObs}}.
\end{proof}
\bibliographystyle{abbrv}
\bibliography{prod}
    \end{document}